\documentclass{amsart}
\usepackage{amssymb}
\usepackage{amsfonts}
\usepackage{amsmath}

\newtheorem{theorem}{Theorem}

\newtheorem{definition}[theorem]{Definition}
\newtheorem{example}[theorem]{Example}

\newtheorem{lemma}[theorem]{Lemma}

\newtheorem{proposition}[theorem]{Proposition}
\newtheorem{remark}[theorem]{Remark}

\DeclareMathOperator{\inte}{int}
\DeclareMathOperator{\card}{card}
\input{tcilatex}

\begin{document}

\title[Skew products, recurrence, shrinking targets]{Skew products, quantitative recurrence, shrinking targets and decay
of correlations}
\author{Stefano Galatolo, J\'er\^ome Rousseau, Benoit Saussol}

\address{Stefano Galatolo, Dipartimento di Matematica Applicata "U. Dini", Universit\'a di Pisa, Via Buonarroti 1, Pisa }
\email{s.galatolo@ing.unipi.it}
\urladdr{http://users.dma.unipi.it/galatolo/}

\address{J\'er\^ome Rousseau, Departamento de Matem\'atica, Universidade Federal da Bahia\\
Av. Ademar de Barros s/n, 40170-110 Salvador, Brazil}
\email{jerome.rousseau@ufba.br}
\urladdr{http://pageperso.univ-brest.fr/~rousseau}

\address{Benoit Saussol, Universit\'e Europ\'eenne de Bretagne, Universit\'e de Brest, Laboratoire de Math\'ematiques CNRS UMR 6205, 6 avenue Victor le Gorgeu, CS93837, F-29238 Brest Cedex 3, France}
\email{benoit.saussol@univ-brest.fr}
\urladdr{http://www.math.univ-brest.fr/perso/benoit.saussol/}
\thanks{This work was partially supported by the ANR Perturbations (ANR-10-BLAN 0106), Indam, DynEurBraz}
\maketitle

\begin{abstract}
We consider toral extensions of hyperbolic dynamical systems. We prove that
its quantitative recurrence (also with respect to given observables) and
hitting time scale behavior depend on the arithmetical properties of the
extension.

By this we show that those systems have a polynomial decay of correlations
with respect to $C^{r}$ observables, and give estimations for its exponent,
which depend on $r$ and on the arithmetical properties of the system. 

We also show examples of systems of this kind having not the shrinking
target property, and having a trivial limit distribution of return time
statistics.
\end{abstract}

\tableofcontents

\section{Introduction}

\subsection{A short overview of hitting and recurrence time}

The study of hitting time or recurrence time in dynamical systems has several
facets. Many of them are described by the surveys \cite{abadigalves, Sa2, barreira} and
references therein, just to mention a few of these aspects:

\begin{itemize}
\item information theory in the study of repetitions of words in sequences
of symbols \cite{ow},

\item probability in dynamical Borel-Cantelli lemma \cite{GK}, shrinking
targets problems \cite{HV}, exponential law, Poisson statistics and extreme
value theory \cite{F},

\item geometric measure theory in the relation with Hausdorff dimension \cite%
{BS,Ga,gal07}

\item geometry, in logarithm laws for the geodesic flow and similar(see e.g. 
\cite{Sul,KM,GN,Masur})

\item nonlinear analysis in recurrence plots (see e.g. \cite{mrtk}).
\end{itemize}

Quite recent results showed relations between quantitative indicators
representing the scaling behavior of return times and hitting times in small
targets, decay of correlations and arithmetical properties. More precisely,
let us consider a discrete time dynamical system $(X,T,\mu)$, where $(X,d)$
is a metric space and $T:X\rightarrow X$ is a measurable map preserving a
finite measure $\mu $. Let us consider two points in $X$ and the time which
is necessary for the orbit of $x$ to approach $y$ at a distance less
than $r$%
\begin{equation*}
\tau _{r}(x,y)=\min \{n\in \mathbb{N^{+}}:d(T^{n}(x),y)<r\}.
\end{equation*}

We consider the behavior of $\tau _{r}(x,y)$ as $r\rightarrow 0$. In many
interesting cases this is a power law $\tau _{r}(x,y)\sim r^{R}$. When $x=y$
this exponent $R$ gives a quantitative measure of the speed of recurrence of
an orbit near to its starting point, and this will be a quantitative
recurrence indicator. When $x\neq y$ the exponent is a quantitative measure
of how fast the orbit starting from $x$ approaches a point $y$. It is indeed
a measure for the scaling behavior of the hitting time of an orbit starting
from $x$ to a sequence of small targets: the balls $B_r(y)$ centered in $y$
of radius $r$. To extract the exponent the recurrence rate ($R(x,x)$ ) and
hitting time scaling exponent ($R(x,y)$ ) are defined by 
\begin{eqnarray}  \label{expo}
\overline{R}(x,y) &=&\limsup_{r\rightarrow 0}\frac{\log \tau _{r}(x,y)}{%
-\log r},\quad \underline{R}(x,y)=\liminf_{r\rightarrow 0}\frac{\log \tau
_{r}(x,y)}{-\log r}.
\end{eqnarray}

A general philosophy is that in ``chaotic'' systems these scaling exponents
are equal to the local dimension. Indeed, several results relate those
exponents to the scaling behavior of the measure of the target set $B_{r}(y)$%
, the local dimension $d_{\mu }(y)$ of the measure $\mu $ at $y$ (see e.g. 
\cite{Bosh},\cite{BS},\cite{Sa},\cite{SR},\cite{Sa},\cite{su},\cite{G10},%
\cite{Ga},\cite{gal07},\cite{GP10},\cite{G2},\cite{Phi}). More precisely: in 
\cite{Sa} (see also \cite{SR}) and \cite{gal07} it is proved that if a
system has superpolynomial decay of correlations with respect to Lipschitz
observables (equivalently for $C^{k}$ observables, see Appendix), then the
recurrence and hitting time exponents are equal to the local dimension of
the invariant measure (see Theorem~\ref{maine} for the precise result).

On the other hand the definition of the above scaling exponents also recalls
diophantine approximation. Indeed in \textquotedblleft non chaotic
systems\textquotedblright\ like Rotations, Interval exchanges or
reparametrizations of rotations, their behavior is related to arithmetical
properties of the system and not to the dimension (see \cite{BC}, \cite{KiM}%
, \cite{kimseo}, \cite{GP??}, \cite{BS} and Section~\ref{rotation} below).

It is worth to remark that there exist mixing systems (with subpolynomial
rate, see \cite{GP??}) where the exponents are not related to dimension but
to arithmetical properties.

\subsection{Outline of the results}

The question remained open, if there are polynomially mixing systems (with
respect to Lipschitz observables) such that the above defined
recurrence-hitting time exponents are different from dimension, and the
question whether this can happen for superpolynomially mixing system with
respect to smoother observables ($C^{\infty }$ observables e.g.).

We consider in this paper a skew product of a chaotic map with a rotation
with some finite Diophantine type (see below and Section~\ref{skew1} for
more precise definitions) and we will show that such a system posses these
properties. To be more precise, we are
interested in systems of the type 
\begin{equation*}
(\omega ,t)\rightarrow (T\omega ,t+\varphi (\omega ))
\end{equation*}%
on a phase space $(\omega ,t)\in \Omega \times \mathbb{R}^{d}/\mathbb{Z}^{d}$
, where $T$ is a piecewise expanding map, and on the second coordinate we
have a family of isometries controlled by a certain function $\varphi $. The
most basic, yet nontrivial example is the case where $T$ is the doubling map
on $S^{1}$, and on the second coordinate we have circle translations.

Ergodic properties of a skew products as above have been quite well studied
in the literature as they are one of the most simple example of systems
which is in some sense weakly chaotic. The qualitative ergodic theory of
these systems has been studied by Brin. A quantitative result about its
speed of mixing was given in \cite{Do}, where it is proved that provided the
rotation's angle has some diophantine properties, such systems have at least
polynomial decay of correlations with respect to Lipschitz and $C^{k}$
observables, moreover they have superpolynomial decay of correlations with
respect to $C^{\infty }$ ones. More recently, in \cite{G07} it was shown
that under suitable assumptions, including the above cited most basic
example, if $\varphi $ is $C^{1}$ and not cohomologous to a piecewise
constant function, then the decay of correlations is exponential (for H\"{o}%
lder observables).

In this paper, we investigate the quantitative recurrence and hitting time
indicators of such kind of systems and its relations with decay of
correlations, in the complementary case where $\varphi $\emph{\ is piecewise
constant. }In this case we show the following facts:

\begin{itemize}
\item The decay of correlations under Lipschitz or $C^{k}$ observables is
polynomial and we give a concrete estimation for its exponent depending on
the arithmetical properties of the system (in the case of multidimensional
rotations the linear diophantine type of the angle is involved, see Sections %
\ref{decorr1}, \ref{decorr2}, \ref{lht}).

\item The recurrence and hitting time exponents are related both to the
dimension and the arithmetical properties of the rotation and we give
estimations for these exponents (see Sections \ref{rec_and_ht},\ref{lht}).

\item We show that if the Diophantine type is large, the statistics of
return time and hitting time has a trivial limiting distribution along some
subsequence. In particular, there is no convergence to the exponential law
for these statistics.

\item Using a multidimensional rotation with a suitable angle (with
intertwined partial quotients) the resulting skew product is an example of a
system with \emph{polynomial} decay of correlations, but not satisfying the
monotone shrinking target property and having a trivial distribution of
return and hitting times. We remark that a mixing system without the
monotone shrinking target property was shown in \cite{B}. However the speed
of decay of correlations of that example is less than polynomial (\cite{GP??}%
), see Section \ref{lht}.
\end{itemize}

We can hence consider these skew products as a borderline case for the
relations between recurrence, hitting time and decay of correlations.

An ingredient of our proofs is an estimate of a discrepancy for random walks
generated by multidimensional irrational rotations. This result (which is a
generalization of the one given in \cite{su}) is proven in the appendix.
There we also prove some other technical results which are more or less
known or at least expected but which are, as far as we know, not present in
this form in the literature.

\section{Background.}

\subsection{General inequalities for hitting and recurrence time exponents}

Recall that the upper and lower local dimension of $\nu $ at $y$ are defined
as $\overline{d}_{\nu }(y)=\underset{r\rightarrow \infty }{\lim \sup }\frac{%
\log \nu (B(y,r))}{\log r}$ and \underline{$d$}$_{\nu }(y)=\underset{%
r\rightarrow \infty }{\lim \inf }\frac{\log \nu (B(y,r))}{\log r}$. It is
relatively easy to obtain that for a general systems (\cite{BS},\cite{Ga}):

\begin{proposition}
\label{GAN}Let $(X,T,\nu )$ be a dynamical system over a separable metric
space and $\nu $ be an invariant Borel measure. For each $y$ 
\begin{equation}
\underline{R}(x,y)\geq \underline{d}_{\nu }(y)\ ,\ \overline{R}(x,y)\geq 
\overline{d}_{\nu }(y)  \label{thm4}
\end{equation}%
holds for $\nu $-almost each $x\in X$. Moreover, if $X$ is a measurable subset of 
$\mathbb{R}^{d}$ for some $d$, then%
\begin{equation}
\underline{R}(x,x)\leq \underline{d}_{\nu }(x)\ ,\ \overline{R}(x,x)\leq 
\overline{d}_{\nu }(x)  \label{easyrec}
\end{equation}%
for $\nu$-a.e. $x\in X$.
\end{proposition}

\subsection{Rapid mixing and consequences}

A sufficient condition for these inequalities to become equalities is that
the system is mixing sufficiently rapidly. This is measured by the decay of
correlations.

\begin{definition}[Decay of correlations]
\label{def:decorr}Let $\phi ,$ $\psi :X\rightarrow \mathbb{R}$ be
observables on $X$ belonging to the Banach spaces $B,B^{\prime }$. Let $\Phi
:\mathbb{N\rightarrow R}$ such that $\Phi (n)\underset{n\rightarrow \infty }{%
\rightarrow }0$.\ A system $(X,T,\nu )$ is said to have decay of
correlations with speed $\Phi $ with respect to observables in $B$ and $%
B^{\prime }$ if 
\begin{equation}
\left\vert \int \phi \circ T^{n}\psi d\nu -\int \phi d\nu \int \psi d\nu
\right\vert \leq \left\vert \left\vert \phi \right\vert \right\vert
_{B}\left\vert \left\vert \psi \right\vert \right\vert _{B^{\prime }}\Phi (n)
\label{decorr}
\end{equation}%
where $||~||_{B}$,$||~||_{B^{\prime }}$ are the norms\footnote{%
It is worth to remark that under reasonable assumptions, a non uniform
statement:$\left\vert \int \phi \circ T^{n}\psi d\mu -\int \phi d\mu \int
\psi d\mu \right\vert \leq C_{\phi ,\psi }\Phi (n)$ \ can be converted in an
uniform one, like \eqref{decorr}. See \cite{CCS}.} in $B$ and $B^{\prime }$.
We say that the decay is polynomial with exponent between $a$ and $b$ if 
\begin{equation*}
a\leq \lim \inf_{n\rightarrow \infty }\frac{-\log \Phi (n)}{\log n}\leq \lim
\sup_{n\rightarrow \infty }\frac{-\log \Phi (n)}{\log n}\leq b.
\end{equation*}%
We say that the decay of correlation is superpolynomial if $\lim n^{\alpha
}\Phi (n)=0,$ $\forall \alpha >0.$
\end{definition}

The decay of correlations has been studied for many systems, mostly those
with some form of hyperbolicity. It depends on the class of observables, and
it is obvious that if a system has decay of correlations with some given
speed with respect to Lipschitz or $C^{p}$ observables, then it has also
with respect to $C^{k}$ ones when $k\geq p$ (with at least the same speed).
There is also a converse which will be used later.

\begin{lemma}
\label{cr}If the space where the dynamics act is a manifold $X\subset\mathbb{%
R}^d$ (possibly with boundary) and the system has decay of correlations with
respect to $C^{p},C^q$ observables with speed $\Phi_{p,q} (n)$ then it has
also with respect to $C^k,C^\ell$ observables with speed $\Phi_{k,\ell} (n)\le\Phi_{p,q}(n)^{%
\frac{1}{\frac{p}{k}+\frac{q}{\ell}-1}}$. In the case $k=1$ or $\ell=1$ the
same is true with Lipschitz (observables and norm) instead of $C^1$.
\end{lemma}

The proof of the above lemma is postponed to the appendix.

Now we can state a result linking decay of correlations and local dimension
with recurrence and hitting time (\cite{gal07},\cite{SR}). We see that in
rapidly mixing systems the recurrence and hitting time exponents are
necessarily equal to the local dimension of the invariant measure.

\begin{theorem}
\label{maine} If $(X,T,\nu )$ has superpolynomial decay of correlations with
respect to Lipschitz observables, as above and the local dimension $d_{\nu
}(y)$ exists\footnote{%
The limit $d_{\nu }(y)=\underset{r\rightarrow \infty }{\lim }\frac{\log \nu
(B(y,r))}{\log r}$ exists.} then%
\begin{equation}
\overline{R}(x,y)=\underline{R}(x,y)=d_{\nu }(y)  \label{maineq}
\end{equation}%
for $\nu $-almost each $x$. If moreover $X\subseteq \mathbb{R}^{d}$ for some 
$d$, then%
\begin{equation*}
\overline{R}(x)=\underline{R}(x)=d_{\nu }(x)
\end{equation*}%
for $\nu $-almost each $x$.
\end{theorem}

\subsection{Basic results on circle rotations}

\label{rotation}

We will see that in the class of system we are interested in, arithmetical
properties are important for quantitative recurrence and hitting time
scaling behavior. The simplest case of system where this happen is the case
of rotations on the circle.

We state some basic results, linking quantitative recurrence, shrinking
targets and arithmetical properties in circle rotations.

\begin{definition}
\label{type}Given an irrational number $\alpha $ we define the \emph{type}
of $\alpha $ as the following (possibly infinite) number: 
\begin{equation*}
\gamma (\alpha )=\inf \{\beta :\liminf_{q\rightarrow \infty }q^{\beta }\Vert
q\alpha \Vert >0\}.
\end{equation*}
\end{definition}

Let $q_{n}$ being the sequence of convergent denominators of $\alpha $ (see
the appendix for some recalls on continued fractions). It is worth to remark
that the above definition is equivalent to the following%
\begin{equation}
\gamma (\alpha )=\limsup_{n\rightarrow \infty }\frac{\log q_{n+1}}{\log q_{n}%
}.
\end{equation}%
Every number has type $\geq 1$. The set of number of type $1$ is of full
measure; the set of numbers of type $\gamma $ has Hausdorff dimension $\frac{%
2}{\gamma +1}$. There exist numbers of infinite type, called \emph{Liouville}
numbers; their set is dense and uncountable and has zero Hausdorff dimension.

The behavior of quantitative recurrence and hitting time in small targets
for circle rotations with angle $\alpha $ depend on the type of $\alpha $,
as it is described in the following statement

\begin{theorem}
\label{KS...}(\cite{kimseo}, \cite{BS}) If $T_{\alpha }$ is a rotation of
the circle, $y$ a point on the circle then for almost every $x$ 
\begin{equation*}
\overline{R}(x,y)=\gamma (\alpha ),\qquad \underline{R}(x,y)=1.
\end{equation*}%
and for each $x$%
\begin{equation}
\underline{R}(x,x)=\frac{1}{\gamma (\alpha )},\qquad \overline{R}(x,x)=1.
\end{equation}
\end{theorem}

It is interesting to see that the statement implies the impossibility to
construct a system with 
\begin{equation}
\underline{R}(x,y)>d_{\nu }(y)  \label{uneqq}
\end{equation}%
for typical $x$ with a circle rotation, and it is worth to remark that
similarly such an example cannot be constructed by an interval exchange \cite%
{BC}. The situation is different if we can consider two dimensional
rotations (see Section \ref{lht}). We will exploit this and construct a
system having even polynomial decay and still satisfying \eqref{uneqq}.

\subsection{Statistics of hitting/return time}

Let us consider a family of centered balls $B_{r}=B_{r}(y)$ $\ $such that $%
\nu (B_{0})=0$ and $\nu (B_{r})\neq 0$ for $r>0$. We will consider the
statistical distribution of return times in these sets. We say that the
return time statistics of the system converges to $g$ for the balls $B_{r}$,
if%
\begin{equation}
\underset{r\rightarrow 0}{\lim }\frac{\nu (\{x\in B_{r},\tau _{r}(x,y)\geq 
\frac{t}{\nu (B_{r})}\})}{\nu (B_{r})}=g(t).  \label{1221}
\end{equation}

If $g(t)=e^{-t}$ we say that the system has an exponential return time limit
statistics. Such statistics can be found in several systems with some
hyperbolic behavior in some class of decreasing sets, however other limit
distributions are possible.

Starting from all the space instead of the ball $B_{r}$ defines the hitting
time statistics: 
\begin{equation}
\underset{r\rightarrow 0}{\lim }\nu (\{x\in X,\tau _{r}(x,y)\geq \frac{t%
}{\nu (B_{r})}\})=g(t).
\end{equation}%
We will show that in our system there is not exponential statistics, even
more, the limiting distribution can be trivial.

\section{The class of skew products under consideration\label{skew1}}

\subsection{Definition}

We will consider a class of systems constructed as follows. The base is a
measure preserving system $(\Omega ,T,\mu )$. We assume that $T$ is a
piecewise expanding Markov map on a finite-dimensional Riemannian manifold $%
\Omega $, that is:

\begin{itemize}
\item there exists some constant $\beta >1$ such that $\Vert D_{x}T^{-1}\Vert
\leq \beta^{-1}$ for every $x\in \Omega $.

\item There exists a collection $\mathcal{J}=\{J_1,\ldots,J_p\}$ such that
each $J_i$ is a closed proper set and

(M1) $T$ is a $C^{1+\eta}$ diffeomorphism from $\inte J_i$ onto its image;

(M2) $\Omega =\cup_i J_i$ and $\inte J_i\cap \inte J_j=\emptyset$ unless $%
i=j $;

(M3) $T(J_i)\supset J_j$ whenever $T(\inte J_i)\cap \inte J_j\neq\emptyset$.
\end{itemize}

$\mathcal{J}$ is called a Markov partition. It is well known that such a
Markov map is semi-conjugated to a subshift of finite type. Without loss of
generality we assume that $T$ is topologically mixing, or equivalently that
for each $i$ there exists $n_i$ such that $T^{n_i}J_i=\Omega$. We assume
that $\mu$ is the equilibrium state of some potential $\psi\colon\Omega\to 
\mathbb{R}$,  H\"older continuous in each interior of
the $J_i$'s. The sets of the form $J_{i_0,\ldots,i_{q-1}}:=%
\bigcap_{n=0}^{q-1} T^{-n}J_{i_n}$ are called cylinders of size $q$ and we
denote their collection by $\mathcal{J}_q$.

In this setting it is well known that the pointwise dimension $d_{\mu }(x)$
exists and $d_{\mu }(x)=d_{\mu }$, $\mu$-almost everywhere.

The system is extended by a skew product to a system $(M,S)$ where $M=\Omega
\times $ $\mathbb{T}^{d}$, $\mathbb{T}^{d}$ is the $d$-torus and $%
S:M\rightarrow M$ defined by%
\begin{equation}
S(\omega ,t)=(T\omega ,t+\mathbf{\alpha }\varphi (\omega ))  \label{skewprod}
\end{equation}%
where $\mathbf{\alpha }=(\alpha _{1},...,\alpha _{d})\in \mathbb{T}^{d}$ and 
$\varphi =1_{I}$ is the characteristic function of a set $I\subset \Omega $
which is an union of cylinders. In this system the second coordinate is
translated by $\mathbf{\alpha }$ if the first coordinate belongs to $I$. We
endow $(M,S)$ with the invariant measure $\nu=\mu \times m$ ($m$ is
the Haar measure on the torus).

We make the standing assumption that 

(NA) for any $u\in \lbrack -\pi ,\pi ]$, the
equation $fe^{iu\varphi }=\lambda f\circ T$, where $f$ is H\"{o}lder (on the subshift) and $\lambda
\in S^{1}$, has only the trivial solutions $\lambda =1$ and $f$ constant.

This is equivalent to the fact that the map $(\omega ,s)\mapsto (T\omega
,s+u\varphi (\omega ))$ is weakly-mixing on $\Omega \times \mathbb{T}$. The
simple case where \textbf{$I$ is a nonempty union of size $1$ cylinders such
that both $I$ and $I^{c}$ contain a fixed point fulfills this assumption}.\footnote{
Any solution has a modulus $|f|$ constant.
The first fixed point gives then $e^{iu}=\lambda$ while the second gives $\lambda=1$. The existence of such $f$ would then contradicts the ergodicity of $T$.}

We will indicate by $\pi _{1}:\Omega \times \mathbb{T}^{d}\rightarrow \Omega 
$ \ and $\pi _{2}:\Omega \times \mathbb{T}^{d}\rightarrow \mathbb{T}^{d}$
the two canonical projections, moreover, on $\Omega \times \mathbb{T}^{d}$
we will consider the sup distance.

We will suppose that $\mathbf{\alpha }$ is of finite Diophantine type. We
introduce two definitions of Diophantine type, which generalize Definition %
\ref{type} to the higher dimensional cases. The notation $\left\vert
\left\vert .\right\vert \right\vert $ will indicate the distance to the
nearest integer vector (or number) in $\mathbb{R}^{d}$ (or $\mathbb{R}$) and 
$|k|=\sup_{0\leq i\leq d}|k_{i}|$ the supremum norm.

\begin{definition}
\label{linapp} The Diophantine type of $\mathbf{\alpha }=(\alpha
_{1},...,\alpha _{d})$ for the \textbf{linear approximation} is 
\begin{equation*}
\gamma _{l}(\mathbf{\alpha })=\inf \{\gamma ~,s.t.\exists c_{0}>0~s.t.\Vert
k\cdot \mathbf{\alpha }\Vert \geq c_{0}|k|^{-\gamma }~\forall 0\neq k\in 
\mathbb{Z}^{d}\mathbb{\}}
\end{equation*}
\end{definition}

\begin{definition}
The Diophantine type of $\mathbf{\alpha }=(\alpha _{1},...,\alpha _{d})$ for
the \textbf{simultaneous approximation} is 
\begin{equation*}
\gamma _{s}(\mathbf{\alpha })=\inf \{\gamma ~,s.t.\exists c_{0}>0~s.t.\Vert k%
\mathbf{\alpha }\Vert \geq c_{0}|k|^{-\gamma }~\forall 0\neq k\in \mathbb{N\}%
}.
\end{equation*}
\end{definition}

\subsection{Upper bound on decay of correlations\label{decorr1}}

We have an explicit upper bound for the rate of decay of correlations.

\begin{proposition}
\label{pro:doc} For Lipschitz observables the rate of decay is $O(n^{-\frac{1%
}{2\gamma }})$ for any $\gamma >\gamma _{l}(\alpha )$.

For $C^p$, $C^q$ observables, the rate of decay is $O(n^{-\frac{1}{2\gamma}%
\max(p,q,p+q-d)})$ for any $\gamma>\gamma_l(\alpha)$.
\end{proposition}

\begin{remark}
We remark that the rate of decay of correlations is \emph{superpolynomial}
for \emph{$C^{\infty }$ observables}.
\end{remark}

The proof of the proposition is based on the following statement

\begin{lemma}
\label{lem:doc} If the observables $A,B$ are respectively of class $C^{p}$
and $C^{q}$ on $M$, $p+q>d$ and $\int_{M}Ad\nu=0$, the correlations satisfy 
\begin{equation*}
C_{n}(A,B):=\int_{M}AB\circ S^{n}d\nu\leq C\Vert A\Vert _{C^{p}}\Vert B\Vert
_{C^{q}}n^{-\ell +\epsilon }
\end{equation*}
for each $\epsilon >0$, where $\ell =\frac{p+q-d}{2\gamma _{l}(\alpha )}$.
\end{lemma}

\begin{proof}[Proof of Proposition~\protect\ref{pro:doc}]
Let $x,y\in\mathbb{N}^*$. By Lemma \ref{cr} we have $\Phi_{p,q}(n) = O(
\Phi_{px,qy}(n)^{\frac{1}{x+y-1}})$. By Lemma~\ref{lem:doc} we know that $%
\Phi_{px,py}(n) = O(n^{-\frac{(px+qy-d)}{2\gamma_l(\alpha)}+\epsilon})$ for
any $\epsilon>0$. Hence $\Phi_{p,q}(n) = O(n^{-\ell+\epsilon})$ with 
\begin{equation*}
\ell = \frac{1}{2\gamma} \frac{px+qy-d}{x+y-1}.
\end{equation*}
Taking the supremum over $x,y\in\mathbb{N}^*$ gives the conclusion.
\end{proof}

The idea of the proof of Lemma~\ref{lem:doc} is to expand in Fourier series
the observables with respect to the second variable $t$ and see that $%
C_{n}(A,B)$ corresponds to an infinite sum of terms which are similar to
correlation integrals of the Fourier coefficients (which are still functions
of $\omega $ ) with respect to a suitable operator.

Let $L$ be the Perron-Frobenius operator for $(\Omega,T,\psi)$ defined for,
say a bounded function $a$, by 
\begin{equation*}
L(a)(\omega )=\sum_{T(\omega ^{\prime })=\omega }e^{\psi(\omega ^{\prime
})}a(\omega ^{\prime }),\quad a\colon \Omega \rightarrow \mathbb{C}.
\end{equation*}%
Without loss of generality we assume that the potential $\psi$ is normalized
in the sense that $L1=1$. In particular this implies that for any measurable
bounded functions $a$ and $b$, and any integer $n$, 
\begin{equation*}
\int_\Omega a b\circ T^n d\mu = \int_\Omega L^n(a)bd\mu.
\end{equation*}

Given $u\in C^{0}(\mathbb{T}^{d},\mathbb{C})$ we denote its Fourier
coefficients by 
\begin{equation*}
\hat{u}(k)=\int_{\mathbb{T}^{d}}e^{-2i\pi <k,t>}u(t)dm(t),\quad k\in \mathbb{%
Z}^{d}
\end{equation*}%
and recall that $u$ is equal to its Fourier series 
\begin{equation*}
u(t)=\sum_{k\in \mathbb{Z}^{d}}\hat{u}(k)e^{2i\pi <k,t>},\quad t\in \mathbb{T%
}^{d}.
\end{equation*}%
Given $u\in \mathbb{R}$ we define the complexified Perron-Frobenius operator 
$L_{u}$ 
\begin{equation*}
L_{u}(a)=L(e^{iu\varphi }a)
\end{equation*}%
where $\varphi $ is the functions involved in the definition of the skew
product, as above. Note that $L_{u}$ is indeed the Perron Frobenius operator
of the potential 
\begin{equation*}
\psi +iu\varphi .
\end{equation*}

For a given $\omega \in \Omega $ we denote the $k-$th Fourier coefficient of 
$A(\omega ,\cdot )$ by $\hat{A}(\omega ,k)$.

\begin{lemma}
The correlations can be expressed by the Fourier expansion of the
observables and the complexified Perron-Frobenius operator: 
\begin{equation*}
C_n(A,B) = \sum_{k\in\mathbb{Z}^d} \int_\Omega L_{2\pi \langle
k,\alpha\rangle}^n(\hat A(\cdot,-k)) \hat B(\cdot,k) d\mu.
\end{equation*}
\end{lemma}

\begin{proof}
We have, using the Fourier expansion of $B$ that (denoting $S_{n}\varphi
(\omega )=\sum_{i=0}^{n-1}\varphi (T^{i}\omega )$) 
\begin{eqnarray*}
C_{n}(A,B) &=&\sum_{k\in \mathbb{Z}^{d}}\int_{\Omega \times \mathbb{T}%
^{d}}A(\omega ,t)\hat{B}(T^{n}\omega ,k)e^{2i\pi \langle k,t+\alpha
S_{n}\varphi (\omega )\rangle }d\mu (\omega )dm(t) \\
&=&\sum_{k\in \mathbb{Z}^{d}}\int_{\Omega }\hat{B}(T^{n}\omega ,k)e^{2i\pi
\langle k,\alpha S_{n}\varphi (\omega )\rangle }\int_{\mathbb{T}%
^{d}}A(\omega ,t)e^{2i\pi \langle k,t\rangle }dm(t)d\mu (\omega ) \\
&=&\sum_{k\in \mathbb{Z}^{d}}\int_{\Omega }\hat{B}(T^{n}\omega ,k)e^{2i\pi
\langle k,\alpha S_{n}\varphi (\omega )\rangle }\hat{A}(\omega ,-k)d\mu
(\omega ) \\
&=&\sum_{k\in \mathbb{Z}^{d}}\int_{\Omega }L_{2\pi \langle k,\alpha \rangle
}^{n}(\hat{A}(\cdot ,-k))\hat{B}(\cdot ,k)d\mu .
\end{eqnarray*}%
This proves the lemma.
\end{proof}

In the above sum, the Fourier coefficients corresponding to large $k$ will
be estimated by the regularity of the observables. To estimate the other
coefficients we use the next proposition.

Strictly speaking, to finish the proof we have to work on the subshift of
finite type generated by the Markov partition. There is nothing essential,
but at some point we will need that the Perron-Frobenius operator preserves
the space of H\"older functions of exponent $\eta$. For convenience we will
keep the same notation $(\Omega,T,\mu)$ for the symbolic dynamical system,
and leave the details to the reader.

\begin{proposition}[\protect\cite{hennion-herve}]
\label{pro:perturbation} There exists two constants $c_0>0$ and $c_1>0$ such
that, as an operator acting on H\"older functions of exponent $\eta$, for
any $n$ and $u\in[-\pi,\pi]$ the operator norm $\||\cdot\||_{\eta,\eta}$
satisfies 
\begin{equation*}
\||L_u^n\||_{\eta,\eta} \le c_0 e^{-c_1 n u^2}.
\end{equation*}
\end{proposition}

\begin{proof}
This property follows classically from quasi-compactness of the
Perron-Frobenius operators \cite{hennion-herve}:

First, by a perturbative argument and since $\varphi $ is not cohomologous
to a constant, there exists $\beta>0$ and some constants $c_0$ and $c_1$
such that if $u\in(-\beta,\beta)$ then $\|| L_u^n\||_{\eta,\eta} \le c_0
e^{-c_1 n u^2}$.

Next whenever $u\neq 0\mod 2\pi$, the operator $L_u$ has no eigenvalues of
modulus (larger than or) equal to one by the assumption (NA), 
hence the spectral radius is smaller than one, which gives $\lim_{n\to\infty} |\|L_u^n\||_{\eta,\eta}^{1/n}<1$.
An uniform exponential contraction follows then by compacity of $%
[-\pi,\pi]\setminus( \beta,\beta)$. Changing if necessary the values of the
constants, the advertised upper bound also holds in the case $u\not\in(-\beta,\beta)$.
\end{proof}

\begin{proposition}
\label{pro:spectralgap} \label{lem:Lkn} For any $\gamma >\gamma _{l}(\alpha
) $ there exists $c_{2}>0$ such that for any non-zero $k\in \mathbb{Z}^{d}$ 
\begin{equation*}
\||L_{2\pi \langle k,\alpha\rangle}^{n}\||_{\eta,\eta}\leq c_{0} \exp
(-c_{2}|k|^{-2\gamma }n).
\end{equation*}
\end{proposition}

\begin{proof}
Let $\gamma>\gamma_l(\alpha)$. By definition of the Diophantine type of $%
\alpha$, we have $\| \langle k,\alpha\rangle \| \ge c|k|^{-\gamma}$ for some
constant $c>0$ and any $k\in\mathbb{Z}^d$. The result follows by Proposition~%
\ref{pro:perturbation}.
\end{proof}

We are now able to prove the main Lemma.

\begin{proof}[Proof of Lemma \protect\ref{lem:doc}]
We split the estimate into two parts, depending whether $k$ is large or not.
Let $N$ be an integer which will be chosen later.

When $k$ is large the estimate is simple: Since $A\in C^{p}$ and $B\in C^{q}$
we have $\Vert \hat{A}(\omega ,k)\Vert _{\infty }\leq \Vert A\Vert
_{p}|k|^{-p}$ and $\Vert \hat{B}(\omega ,k)\Vert _{\infty }\leq \Vert B\Vert
_{q}|k|^{-q}$. Hence 
\begin{equation*}
\left\vert \int_{\Omega }L_{2\pi \langle k,\alpha \rangle }^{n}(\hat{A}%
(\cdot ,-k))\hat{B}(\cdot ,k)d\mu \right\vert \leq \Vert A\Vert _{p}\Vert
B\Vert _{q}|k|^{-p-q}.
\end{equation*}%
Therefore, since the operator $L_{2\pi\langle k,\alpha\rangle}^{n}$ does not
expand the sup-norm, for any integer $N$ we get 
\begin{equation}
\left\vert \sum_{|k|>N}\int_{\Omega }L_{2\pi \langle k,\alpha \rangle }^{n}(%
\hat{A}(\cdot ,-k))\hat{B}(\cdot ,k)d\mu \right\vert \leq C_{1}\Vert A\Vert
_{p}\Vert B\Vert _{q}N^{d-p-q}.  \label{kgrand}
\end{equation}

The case $k=0$ is done independently: Indeed, we have 
\begin{equation*}
\int_{\Omega }\hat{A}(\cdot ,0)d\mu =\int_{\Omega \times \mathbb{T}%
^{d}}A~d\mu \otimes m=0
\end{equation*}%
by hypothesis. Then the usual decorrelation estimate gives that for some $%
C,\lambda >0$ we have 
\begin{equation}
\left\vert \int_{\Omega }L^{n}(\hat{A}(\cdot ,0))\hat{B}(\cdot ,0)d\mu
\right\vert \leq C_{2}\Vert A\Vert _{1}\Vert B\Vert _{\infty }\lambda ^{n}
\label{kzero}
\end{equation}

(recall that $||.||_{1}$ here is the $C^{1}$ norm). On the other hand, for
small $k\neq 0$ we make use of Proposition~\ref{lem:Lkn}:

Observe that$\Vert \hat{A}(\cdot ,k)\Vert \leq \Vert A\Vert _{p}|k|^{1-p}$
(here $||.||$ is the Holder norm). 
\begin{equation*}
\left\vert \int_{\Omega }L_{2\pi \langle k,\alpha \rangle }^{n}(\hat{A}%
(\cdot ,-k))\hat{B}(\cdot ,k)d\mu \right\vert \leq \Vert A\Vert
_{p}|k|^{1-p}c\exp (-c|k|^{-2\gamma }n)\Vert B\Vert _{q}|k|^{-q}.
\end{equation*}%
Therefore 
\begin{equation}
\sum_{1\leq |k|\leq N}\left\vert \int_{\Omega }L_{2\pi \langle k,\alpha
\rangle }^{n}(\hat{A}(\cdot ,-k))\hat{B}(\cdot ,k)d\mu \right\vert \leq
c\Vert A\Vert _{p}\Vert B\Vert _{q}\sum_{1\leq |k|\leq N}|k|^{1-p-q}\exp
(-cn|k|^{-2\gamma }).  \label{kpetit}
\end{equation}%
which is easily seen to be superpolynomially decaying in $n$, if we choose $%
N=n^{1/2\gamma -\epsilon }$, since 
\begin{equation*}
\sum_{1\leq |k|\leq N}|k|^{1-p-q}\exp (-cn|k|^{-2\gamma })\leq
N^{d+1-p-q}\exp (-cnN^{-2\gamma }).
\end{equation*}%
Putting together the three estimates \eqref{kgrand}, \eqref{kzero} and %
\eqref{kpetit} gives 
\begin{equation*}
C_{n}(A,B)\leq const\Vert A\Vert _{p}\Vert B\Vert _{q}n^{-\ell },
\end{equation*}%
for any $\ell < \frac{p+q-d}{2\gamma }$.
\end{proof}

\section{Quantitative recurrence, shrinking targets and arithmetical
properties\label{rec_and_ht}}

In this section we give estimations on the quantitative recurrence and
hitting time exponent of skew products, as defined before. These results
show that both the quantitative recurrence and the hitting time exponent
depend on the arithmetical properties of $\alpha $.

\subsection{Deterministic and random multidimensional rotations}

The quantitative recurrence and hitting time exponents of our system is
influenced by the underlying rotation. Multidimensional rotations have not
been completely investigated from this point of view. We shall start with
some general considerations on recurrence, hitting times and discrepancy on
multidimensional rotations and random walks generated by such rotations.
These results will be used in the following.

Let us consider a rotation $x\rightarrow x+\alpha $ on $\mathbb{T}^{d},$ $%
\alpha $ $=(\alpha _{1},...\alpha _{d})$. It is easy to see that for this
map, the recurrence rates $\underline{R}(x,x)$ and $\overline{R}(x,x)$ do
not depend on $x$, hence there are constants that we denote by 
\begin{equation*}
\underline{R}(x,x)=\underline{rec}(\alpha ),\quad \overline{R}(x,x)=%
\overline{rec}(\alpha ).
\end{equation*}%
Reproducing the proof of \cite{BS} for one dimensional rotations, it is easy
to see that

\begin{proposition}
Let us consider a rotation $x\rightarrow x+\alpha $ on $\mathbb{T}^{d},$ $%
\alpha $ $=(\alpha _{1},...\alpha _{d})$. We have 
\begin{equation*}
\underline{rec}(\alpha)=\frac{1}{\gamma_s(\alpha)}.
\end{equation*}
\end{proposition}

Concerning the hitting time exponents $\underline{R}(x,y)$ and $\overline
R(x,y)$ a similar thing happens (using that $ R(x,y)=
R(x-y,0)$ for any $x,y$, together with the invariance $R(x+\alpha,0)=R(x,0)$%
): there exists also two constants $\underline{hit}(\alpha)$ and $\overline{%
hit}(\alpha)$ such that for any $y$ and a.e. $x$ 
\begin{equation*}
\underline R(x,y) = \underline{hit}(\alpha), \quad \overline R(x,y) = 
\overline{hit}(\alpha).
\end{equation*}

The following proposition is quite similar to the transference principle
between homogeneous and inhomogeneous diophantine exponents (see \cite%
{cassel}).

\begin{proposition}
\label{pro:alphahit} Let us consider a rotation $x\rightarrow x+\alpha $ on $%
\mathbb{T}^{d},$ $\alpha $ $=(\alpha _{1},...\alpha _{d})$. We have 
\begin{equation*}
\overline{hit}(\alpha)\geq \gamma _{l}(\alpha ).
\end{equation*}
\end{proposition}

\begin{proof}
Let $\gamma <\gamma _{l}(\alpha )$. Then there is a sequence of vectors $%
k_{p}=(k_{1,p},...,k_{d,p})$ with $|k_{p}|\rightarrow \infty $ such that $%
||k_{p}\cdot \alpha ||<|k_{p}|^{-\gamma }.$ By choosing a subsequence we can
also suppose without loss of generality that $p\leq \frac{1}{2}\sqrt{\log
(|k_{p}|)}$.

Let $r_{p}=\frac{|k_{p}|^{-1}}{2p^{2}}$. For our purpose it is sufficient to
prove that for a full measure set of $t\in\mathbb{T}^d$, eventually (as $%
p\rightarrow \infty $) 
\begin{equation*}
||n\alpha _{1}+t_{1}||<r_{p},...,||n\alpha _{d}+t_{d}||<r_{p}\implies n>%
\frac{d|k_{p}|^{\gamma }}{\log |k_{p}|}.
\end{equation*}

Let us fix $p$ and choose $t=(t_{1},...,t_{d})$ such that 
\begin{equation*}
||t_{1}k_{1,p}+t_{2}k_{2,p}+...+t_{d}k_{d,p}||>\frac{d}{p^{2}}.
\end{equation*}

Now suppose $n$ is such that $||n\alpha _{1}+t_{1}||<r_{p},...,||n\alpha
_{d}+t_{d}||<r_{p}$ and remark that 
\begin{eqnarray*}
\frac{d}{p^{2}} &<&||k_{p}\cdot t||=||k_{p}\cdot (-n\alpha )+k_{p}\cdot
(t+n\alpha )|| \\
&\leq &n||k_{p}\cdot \alpha ||+d|k_{p}|||t+n\alpha || \\
&\leq &n|k_{p}|^{-\gamma }+d|k_{p}|r_{p} \\
&\leq &n|k_{p}|^{-\gamma }+\frac{d}{2p^{2}}
\end{eqnarray*}%
and thus $n\geq \frac{d|k_{p}|^{\gamma }}{2p^{2}}\geq \frac{d|k_{p}|^{\gamma
}}{\log |k_{p}|}$.

Now, it remains to show that the set of $t$ such that $||t\cdot k_p||>\frac{d%
}{p^{2}}$ eventually has full measure.

Indeed, let us consider the $\mathbb{Z}^{d}$-periodic set 
\begin{equation*}
C_{p}=\{t\in \mathbb{R}^{d}~s.t.~||t\cdot k_{p}||\leq \frac{d}{p^{2}}\}.
\end{equation*}%
Now, choose an integer nonsingular matrix such that $Ae_{1}=k_{p}$ ($e_{1}$
is the first vector of the canonical basis of $\mathbb{R}^{d}$) then $%
C_{p}=\{t\in \mathbb{R}^{d}~s.t.~||A^{t}t\cdot e_{1}||\leq \frac{d}{p^{2}}%
\}. $ Now the measure of these sets must be evaluated on the torus $T^{d}=%
\mathbb{R}^{d}/\mathbb{Z}^{d}.$ Let us denote by $\pi $ the canonical
projection $\mathbb{R}^{d}\rightarrow T^{d}$. Let us now consider the set $%
C_{p}^{\prime }=\{x\in \mathbb{R}^{d}~s.t.~||x\cdot e_{1}||\leq \frac{d}{%
p^{2}}\}.$ First let us remark that $\mathrm{Leb} (\pi (C_{p}^{\prime }))=%
\frac{2}{p^{2}}$. Now remark that being a nonsingular integer matrix, $A^{t}$
induces a measure preserving map $\overline{A}^{t}$on the torus and $\pi
(C_{p})=(\overline{A}^{t})^{-1}[\pi (C_{p}^{\prime })]$, and then $\mathrm{%
Leb} (\pi (C_{p}))=\frac{2}{p^{2}}$. \ This allows us to apply Borel Cantelli
lemma and deduce that the set of points belonging to $C_{p}$ for infinitely
many values of $p$ is a zero measure set.
\end{proof}

We recall the classical notion of discrepancy for a rotation of angle $%
\alpha\in\mathbb{R}^d $. Let $\mathcal{P}^d$ be the set of rectangles $%
R=\prod_i[a_i,b_i]\subset [0,1]^d$. Let $|R|=\prod_i(b_i-a_i)$. 
\begin{equation*}
D_{n}(\alpha )=\sup_{R\in \mathcal{P}^d}\left |\frac{1}{n}\card\left\{0\leq
k\leq n-1\colon k\alpha \in R+\mathbb{Z}^d\right\}-|R|\right|.
\end{equation*}%
It is governed by the linear Diophantine type of $\alpha $

\begin{proposition}[\protect\cite{niederreiter}]
\label{discrepancy} The discrepancy satisfies, for any $\gamma>\gamma_l(%
\alpha) $, 
\begin{equation*}
D_n(\alpha) = O\left( n^{-1/\gamma} \right).
\end{equation*}
\end{proposition}

In addition, the second coordinate of the skew product evolves as a random
rotation, thus we will need a notion of discrepancy of the random walk on
the torus $\alpha S_{n}\varphi $ driven by the Gibbs measure $\mu $, defined
by 
\begin{equation*}
D_{n}^{\mu }(\alpha )=\sup_{R\in\mathcal{P}^d}\left|\mu \left(\left\{\omega
~|~\alpha S_{n}\varphi (\omega )\in R+\mathbb{Z}^d\right\}\right)-|R|\right|.
\end{equation*}%
In this random setting the upper bound become

\begin{proposition}[After \protect\cite{su}]
\label{pro:dnmu} The discrepancy satisfies, for any $\gamma>\gamma_l(\alpha)$%
, 
\begin{equation*}
D_n^\mu(\alpha) = O\left( (\sqrt{n})^{-1/\gamma} \right).
\end{equation*}
\end{proposition}

The proof is postponed to the appendix.

\subsection{Hitting time}

We start with an easy consequence of the proof of Proposition~\ref%
{pro:alphahit}.

\begin{theorem}
\label{thm:trivialHTS} Suppose that $\gamma_l(\alpha)>\overline{d}%
_\mu(\pi_1(y))+d$. Then, there exists a sequence $r_n\to0$ such that the
hitting time statistics to balls $B(y,r_n)$ has a trivial limiting
distribution.
\end{theorem}

In particular these systems are polynomially mixing but the distribution of
hitting time to balls does not converge to the exponential law.

\begin{proof}
We assume without loss of generality that $\pi_2(y)=0$. Let $%
\gamma<\gamma_l(\alpha)$ and $\epsilon>0$ such that $\gamma-2\epsilon>%
\overline{d}_\mu(\pi_1(y))+d$. In the proof of Proposition~\ref{pro:alphahit}
it is shown that there exists a sequence $r_p\to0$ such that the hitting
time to the $r_p$-neighborhood of $\mathbb{Z}^d$ by the sequence $t+n\alpha$
is at least $r_p^{-\gamma+\epsilon}$, for all $t\in\mathbb{T}^{d}$ outside a
set of measure $2/p^2$. Since $S_n\varphi\le n$, this implies that $%
\tau_{r_p}(x,y)\ge r_p^{-\gamma+\epsilon}$. Let $s>0$. For any $r$
sufficiently small $\frac{s}{\nu(B(y,r))}<r_p^{-\gamma+\epsilon}$.
Therefore, for any $s>0$, 
\begin{equation*}
\nu\left(x\colon \tau_r(x,y) \le \frac{s}{\nu(B(y,r))}\right) \le \frac2{p^2}
\end{equation*}
provided $p$ is sufficiently large.
\end{proof}

The preceding result allows us to estimate the hitting time exponents in our
skew product.

\begin{theorem}
\label{11}In the skew product $(M,S,\nu)$ described above, at each target
point $y\in M$ it hold 
\begin{equation}
\overline{R}(x,y)\geq \max (\overline{d}_{\mu }(\pi _{1}(y))+d,\overline{hit}%
(\alpha)) \quad\text{and}\quad \underline{R}(x,y)\geq \max (\underline{d}%
_{\mu }(\pi _{1}(y))+d,\underline{hit}(\alpha))  \label{XX}
\end{equation}%
for $\nu$-a.e. $x\in M$ (recall that $\pi _{i}$ are the two canonical
projections, as defined above).
\end{theorem}

\begin{proof}
The inequality $\overline{R}(x,y)\geq \overline{d}_{\mu }(\pi _{1}(y))+d$
follows for $\nu$-a.e. $x$ from the general inequality \eqref{thm4}, since the
invariant measure is the product measure of $\mu $ and the Haar measure on
the torus, then $\overline{d}_{\mu \times m}(y)=\overline{d}_{\mu }(\pi
_{1}(y))+d$. In particular $\tau_r(x,y)\to\infty$ as $r\to0$.

Furthermore, for $\nu$-a.e. $x=(\omega,t)$, the upper hitting time exponent $%
R(t,0)$ for the rotation by angle $\alpha$ is equal to $\overline{hit}%
(\alpha)$. Let $h>\overline{hit}(\alpha)$. There exists a sequence $r_p\to0$
such that $\|t+n\alpha\|>r_p$ for $n=1,\ldots,\lfloor r_p^{-h}\rfloor$.
Since $0<S_n\varphi(\omega)\le n$ for any $n$ sufficiently large and $%
\tau_{r_p}(x,y)\to\infty$, we get that $\tau_{r_p}(x,y)\ge r_p^{-h}$ for any 
$p$ sufficiently large. hence $\overline{R}(x,y)\ge h$. The first inequality
follows since $h$ is arbitrary.

The corresponding statement for $\underline{R}$ can be done with a
simplification of the same proof (since now the estimates will be valid for
any $r$ and not only for a sequence $r_p$).
\end{proof}

The question if the above general inequality in Theorem~\ref{11} is sharp
arises naturally. We show that if $\mu $ is a Bernoulli measure and $d=1$
this is the case.

\begin{proposition}
We assume that all the branches of the Markov map $T$ are full, i.e. $%
T(J_{i})=\Omega $ for all $i$, that $\mu $ is a Bernoulli measure i.e. $\mu
([a_{1}\ldots a_{n}])=\mu ([a_{i}])\cdots \mu ([a_{n}])$, and $I$ depends
only on the first symbol, i.e. $I$ is an union of $1$-cylinders (recall that 
$\varphi =1_{I}$).

Then for any $y$, for $\nu$-a.e. $x$ we have 
\begin{equation*}
\overline{R}(x,y)\leq \max (\overline{d}_{\mu }(\pi _{1}(y))+d,d\gamma
_{l}(\alpha )).
\end{equation*}
\end{proposition}

\begin{remark}
Under the assumptions of the proposition:

\begin{itemize}
\item If $\gamma _{l}(\alpha )>d_{\mu }+d$ then the hitting time exponent is
larger than the dimension of the invariant measure a.e.

\item In dimension $d=1$ we get the equality for all $y$ 
\begin{equation*}
\overline{R}(x,y)=\max (d_{\mu }(\pi _{1}(y))+1,\gamma(\alpha )),
\end{equation*}%
for $\nu$-a.e. $x$.
\end{itemize}
\end{remark}

\begin{proof}
Without loss of generality we assume that $y=(\omega^{\prime },0)$.

Let $\xi =\max (\overline{d}_{\mu }(\omega ^{\prime })+d,d\gamma
_{l}(\alpha))$ and let $\epsilon>0$. Let $\gamma=\gamma_l(\alpha)+\epsilon/d$%
.

For $r>0$ set $N_r=r^{-\xi-\epsilon}$. For all $t$ we have by Proposition~%
\ref{discrepancy} that $D_{N_r}(\alpha)\le N_r^{-\frac1\gamma}$ whenever $r$
is sufficiently small, hence 
\begin{equation}  \label{eq:phir}
\begin{split}
\card\{k\le N_r\colon \|t+k\alpha\|<r\} &\ge
N_r\left((2r)^d-D_{N_r}(\alpha)\right) \\
&\ge r^{-\xi-\epsilon}( 2^d r^d - r^{\frac{\xi+\epsilon}{\gamma}} ) \\
&\ge r^{d-\xi-\epsilon}.
\end{split}%
\end{equation}

Denote by $K_r=\{k_r^1<k_r^2<\cdots<k_r^p\}$ the set of integers $k<N_r$
such that $\|t+k\alpha\|<r$. Notice that $k_r^1$ is the first hitting time
of $[-r,r]$ by the orbit of $t$ under the rotation of angle $\alpha$,
therefore by Theorem~\ref{KS...}, for a.e. $t$ we have $\liminf_{r\to 0}\log
k_r^1/\log(1/r)\ge 1$. In particular for any $r$ sufficiently small, 
\begin{equation}  \label{eq:kr1}
k_r^1 \ge -\log r.
\end{equation}

We now fix some typical $t$ and consider $r$ so small that \eqref{eq:phir}
and \eqref{eq:kr1} hold. Let $b\in(0,\mu(A))$.

Let $G=\{\omega\in\Omega\colon \forall n=k_r^1,...,N_r,
S_n\varphi(\omega)>bn\}$. Let $q$ be the minimal integer such that any $q$%
-cylinder has a diameter less than $r/2$. Note that $q=O(-\log r)$. Let $B$
be the union of all $q$-cylinders that intersect the ball $B(\omega^{\prime
},r/2)$. When $\omega\in G$, if $S_n\varphi(\omega)\in K$ and $%
T^n(\omega)\in B$ then $\tau_r(\omega):=\tau_r((\omega,t),(\omega^{\prime
},0))\le n\le N_r/b$.

Let $t_j(\omega)$ be the smallest integer $n$ such that $S_n\varphi(%
\omega)=k_j$. This means $S_n\varphi(\omega)=k_j>S_{n-1}\varphi(\omega)$.
Since the increments of $S_n\varphi$ are $0$ or $1$ the sets $%
\Gamma_{(n_j)}:=\{t_1=n_1,\ldots,t_p=n_p\}$, $n_1<\cdots<n_p$, form a
partition of $\Omega$. Therefore 
\begin{equation*}
\mu(\tau_r > \frac1b N_r) \le \mu(\Omega\setminus G) + \underbrace{
\sum_{n_1<\cdots<n_p}\mu(\Gamma_{(n_j)}\cap \bigcap_{j=1}^p T^{-n_j}B^c)}%
_{(\star)}.
\end{equation*}

A large deviations estimates show that there exists $c\in(0,1)$ such that $%
\mu(S_n\varphi >bn)\le c^n$ for any $n$ sufficiently large. With %
\eqref{eq:kr1} this gives the estimate for the first term 
\begin{equation*}
\mu(\Omega\setminus G)\le \sum_{n=k_r^1}^{N_r} \mu(S_n\le bn) \le
c^{k_r^1}/(1-c) \le r^{-\log c}/(1-c).
\end{equation*}

We estimate the second term $(\star)$. Set $n_1^{\prime }=n_1, n_2^{\prime
}=n_2-n_1,\ldots,n_p^{\prime }=n_p-n_{p-1}$ and define $(k_j^{\prime })$ in
the same way. The set $\Gamma_{(n_j)}\cap \bigcap_{j}T^{-n_j}B^c$ is equal
to the intersection of the sets 
\begin{equation*}
\begin{split}
& \{ S_{n^{\prime }_1}\varphi=k^{\prime }_1 > S_{n^{\prime }_1-1}\varphi\},
\\
& T^{-n_1}(B^c\cap \{S_{n^{\prime }_2}\varphi=k^{\prime }_2 > S_{n^{\prime
}_2-1}\varphi \}), \\
& \vdots \\
& T^{-n_{p-1}}(B^c\cap \{S_{n^{\prime }_p}\varphi=k^{\prime }_p >
S_{n^{\prime }_p-1}\varphi \}), \\
& T^{-n_p}B^c.
\end{split}%
\end{equation*}
Since $\varphi$ depends only on the first symbol, these sets depend on
different coordinates. Thus by the Bernoulli property the measure of the
intersection is the product of the measures: 
\begin{equation*}
(\star)=\mu(\{S_{n^{\prime }_1}\varphi=k^{\prime }_1>S_{n^{\prime }_1-1}
\varphi \}) \prod_{j=2}^{p-1}\mu( B^c \cap \{S_{n^{\prime
}_p}\varphi=k^{\prime }_p>S_{n^{\prime }_p-1}\varphi, \}) \mu(B^c).
\end{equation*}
Since for any $j=1,\ldots,p$ we have 
\begin{equation*}
\sum_{n^{\prime }_j=1}^\infty \mu( B^c \cap \{S_{n^{\prime
}_j}\varphi=k^{\prime }_j>S_{n^{\prime }_j-1}\varphi \}) = \mu(B^c)=1-\mu(B).
\end{equation*}
A summation over $n_1<\cdots<n_p$ then gives 
\begin{equation*}
(\star) = (1-\mu(B))^p.
\end{equation*}
By \eqref{eq:phir} we have $p\ge r^{d-\xi-\epsilon}$. Whenever $r$ is
sufficiently small we have 
\begin{equation*}
\mu(B)\ge \mu(B(\omega^{\prime \overline{d}_\mu(\omega^{\prime
})+\epsilon/2}.
\end{equation*}
Therefore 
\begin{equation*}
(1-\mu(B))^p \le \left(1-r^{\overline{d}_\mu(\omega^{\prime })+\frac\epsilon
2}\right)^{\left[r^{-(\overline{d}_\mu(\omega^{\prime })+\frac \epsilon 2)}
r^{-\frac\epsilon 2}\right]} \le \exp(-r^{-\frac\epsilon2}).
\end{equation*}

Setting $r_n=e^{-n}$ we get that $\sum_n \mu(\tau_{r_n}>\frac1b N_{r_n})
<\infty$. By Borel-Cantelli we conclude that for $\mu$-a.e. $\omega$, if $n$
sufficiently large then 
\begin{equation*}
\tau_{r_n}((\omega,t),(\omega^{\prime },0)) \le \frac 1b r_n^{-\xi-\epsilon}.
\end{equation*}
This implies that $\overline R((\omega,t),(\omega^{\prime },0))\le
\xi+\epsilon$. The conclusion follows since $\varepsilon$ is arbitrary.
\end{proof}

A bound for \underline{$R$}$(x,y)$ similar to the one given in Proposition %
\ref{11} also hold. Since this is related with the dynamical Borel-Cantelli
lemma and similar topics we postopone it to Section \ref{lht}.

\subsection{Quantitative recurrence}

\begin{theorem}
\label{thm:upper} We have $\underline{R}(x,x)\leq d_{\mu }+\frac{1+d_{\mu }}{%
\gamma _{s}(\alpha )}$ for $\nu $-a.e. $x$.
\end{theorem}

\begin{proof}
Since the recurrence does not depend on its initial value, we fix arbitrarily some $t\in\mathbb{T}^d$. Let $\epsilon>0$, set $\gamma=\gamma
_{s}(\alpha)-\epsilon$ and take $\delta=d_{\mu}+2\epsilon$.

We consider a set $K$ with $\mu (K)>1-\epsilon$ arbitrarily close to $1$,
such that for any $r $ there exists a cover of $K$ by balls of diameter $r $
with cardinality less than $c r ^{-d_\mu-\epsilon}$ for some constant $c>0$
(see e.g. \cite{BS} for a detailed construction).

Fix some $q\in\mathbb{N}^*$ such that $\Vert q\alpha \Vert <q^{-\gamma }$. At
this time $q$ the rotation makes by definition a close return to the origin.
The idea is that many of its multiples $kq$ will inherit this property. More
precisely, let $r=r(q)=q^{-\frac\gamma{1+\delta}} $. For any integer $k<r
^{-\delta }$ we still have 
\begin{equation*}
\Vert kq\alpha \Vert \leq k\Vert q\alpha \Vert \leq r ^{-\delta }q^{-\gamma }=r
^{-\delta +\delta+1 }=r .
\end{equation*}%
Since $S^{n}(\omega ,t)=(T^{n}(\omega ),t+\alpha S_{n}\varphi(\omega ))$, if
an integer $n$ satisfies (i) $S_{n}\varphi(\omega )\leq qr ^{-\delta }$, (ii) $%
S_{n}\varphi(\omega )$ is a multiple of $q$ and (iii) $n$ is a $r $-return
for $T$ then $n$ is an $r $-return for $S$.

Let $F$ be the finite extension of $T$ defined on $\Omega \times \mathbb{Z}%
_{q}$ by 
\begin{equation*}
F(\omega ,z)=(T(\omega ),z+\varphi (\omega )).
\end{equation*}%
(we put the discrete metric on $\mathbb{Z}_{q}$). This map preserves the
probability measure $\mu \times \frac{1}{q}H^{0}$ where $H^{0}$ is the
counting measure on $\mathbb{Z}_{q}$.

Note that if $n\leq qr ^{-\delta }$ is an $r $-return for $F$ then $%
S_{n}\varphi(\omega )=0\mod q$, therefore $n$ is also an $r $-return for $S$%
. Therefore if $\tau _{r }^{F}(\omega ,z)$ denotes the $r$-return for $F$
and $\tau _{r }^{S}(\omega ,t)$ is the $r$-return for $S$ then $\tau _{r
}^{S}(\omega ,t)\le \tau _{r }^{F}(\omega ,z)$ whenever the latter is less
than $qr^{-\delta}$.

We now take a cover of $K$ by balls $\{B_i\}$ of diameter $r$ which has the
properties mentioned before.

We have, repeating the computation in \cite{BS} (in particular using K\v{a}%
c's lemma for $F$) 
\begin{equation*}
\begin{split}
\mu(\omega\in K\colon \tau_r^S(\omega,t)>qr^{-\delta}) &\le
\mu\times\frac1{q}H^0((\omega,z)\colon \tau_r^F(\omega,z)>qr^{-\delta}) \\
&\le \sum_i \sum_{j\in\mathbb{Z}_q} \mu\times\frac1{q}H^0((\omega,z)\in
B_i\times\{j\}\colon \tau_{B_i\times \{j\}}^F(\omega,z)>qr^{-\delta}) \\
&\le \sum_i \sum_{j\in\mathbb{Z}_q} \frac{r^\delta}{q} \int_{B_i\times
\{j\}}\tau_{B_i\times \{j\}}^F d\mu\times \frac1{q}H^0 \\
&\le c r^{\epsilon}.
\end{split}%
\end{equation*}

We now take a sequence $q_n\in\mathbb{N}^*$ going to infinity such that $%
\|q_n\alpha\|\le q_n^{-\gamma}$ and $\sum_n r(q_n)^\epsilon<\infty$. By
Borel-Cantelli we get that for $\mu$-a.e. $\omega\in K$ we have 
\begin{equation}  \label{eq:taurqn}
\tau _{r(q_n) }(\omega ,t)\leq r(q_n) ^{-\delta -\frac{1+\delta }{\gamma }}
\end{equation}
for any $n$ sufficiently large. Thus, writing $x=(\omega,t)$ we get 
\begin{equation*}
\underline{R}(x,x)\leq \delta +\frac{1+\delta}{\gamma },
\end{equation*}%
for $\nu$-a.e. $x\in K\times\mathbb{T}^d$. The conclusion follows letting $%
\epsilon\to0$.
\end{proof}

We remark in particular that when $\gamma _{s}$ is large, the lower
quantitative recurrence exponent becomes smaller than the dimension of the
measure $\nu$, which is $d_{\mu }+d$.

The fact that this rapid recurrence occurs at the \emph{same scale} for $\mu$%
-a.e. points enables us to deduce the following result.

\begin{theorem}
\label{thm:trivialRTS} Suppose that $\gamma_s(\alpha)>\frac{1+d_\mu}{d} $.

For $\nu$-a.e. $x$, there exists a subsequence $r_n\to0$ such that the
return time statistics in balls $B(x,r_n)$ has a trivial limiting
distribution.
\end{theorem}

In particular, these systems are polynomially mixing but the return time
distribution to balls does not converge to the exponential law.

\begin{proof}
We use the same notation of the proof of Theorem~\ref{thm:upper}. Consider $%
\epsilon>0$ so small that $\delta +\frac{1+\delta}{\gamma }+3\epsilon<d_\mu+d$.
Take $n_0$ so large that the set $H$ of points $\omega\in\Omega$ such that %
\eqref{eq:taurqn} holds for any $n>n_0$ has a measure $\mu(H)>1-\epsilon$.

Let $\omega\in\Omega$ be such that for any $r$ sufficiently small $%
r^{d_\mu+\epsilon}\le \mu(B(\omega,r))<r^{d_\mu-\epsilon}$ and $%
\mu(B(\omega,r)\cap H)/\mu(B(\omega,r))\to 1$ as $r\to0$. The fact that the
pointwise dimension is $\mu$-a.e. equal to $d_\mu$ and Lebesgue density
theorem shows that this concerns $\mu$-a.e. points of $H$.

Given $r>0$ we set $L_r=\lceil \log^2 r\rceil$. For any $r$ sufficiently
small we have 
\begin{equation*}
\frac{\mu(B(\omega,2L_r r))}{\mu(B(\omega,L_r r))} \le r^{-3\epsilon}.
\end{equation*}
Hence there exists an integer $k_r\in[L_r,2L_r]$ such that 
\begin{equation}  \label{eq:corona}
\mu(B(x,k_rr))-\mu(B(x,(k_r-1)r))\le \frac{1}{-\log r}\mu(B(x,k_r r)),
\end{equation}
otherwise 
\begin{equation*}
\mu(B(\omega,2L_r r)) (1-\frac1{-\log r})^{L_r} \ge \mu(B(\omega,L_r r))
\end{equation*}
which would contradict the previous inequality provided $r$ is sufficiently
small.

Let $r_n=k_{r(q_n)}r(q_n)$. If $\omega^{\prime }\in H\cap B(x,r_n)\setminus
B(x,r_n-r(q_n))$ we have by \eqref{eq:taurqn} that 
\begin{equation*}
\tau_{r_n}(\omega^{\prime },t)\le r(q_n)^{-d_\mu+d-3\epsilon}\le
\mu(B(x,r_n))r_n^\epsilon
\end{equation*}
for any $n$ sufficiently large. For any $s>0$, once $r_n\le s$ we get,
setting $x=(\omega,t)$, 
\begin{equation*}
\nu_{B(x,r_n)}(y\colon \tau_{r_n}(y)\ge \frac{s}{\mu(B(x,r_n))}) \le
\mu_{B(\omega,r_n)}(H^c\cap B(\omega,r_n-r(q_n))).
\end{equation*}
Using \eqref{eq:corona} and the fact that $\omega$ is a Lebesgue density
point of $H$ we conclude that the upper bound goes to zero as $n\to\infty$.
\end{proof}

\begin{theorem}
In a skew product as above the lower recurrence rate is bounded from below
by 
\begin{equation*}
\underline{R}(x,x)\geq \min\left(\frac{d_{\mu }}{1-\frac{1}{2\gamma_l
(\alpha )}}, d_\mu+d\right) \quad\nu -a.e.x.
\end{equation*}
\end{theorem}

\begin{proof}
Let $\gamma>\gamma_l(\alpha)$. Let $\epsilon>0$. Set $\Delta=\min\left((d_%
\mu-3\epsilon)/(1-1/2\gamma),d_\mu+d\right)$.

We take a set $K$ of measure $\mu(K)>1-\epsilon$ and $r_0>0$ such that for
any $\omega\in K$ and $r\in(0,r_0)$ we have $\mu(B(\omega,2r))\le
r^{d_\mu-\epsilon}$.

Given $r>0$ we fix $k$ as the smallest integer such that any cylinder $Z\in%
\mathcal{J}_k$ has a diameter less than $r$. Let $Z\in \mathcal{J}_{k}$. Let 
$B(Z,r)$ be the union of balls $\bigcup_{\omega^{\prime }\in
Z}B(\omega^{\prime },r)$. Hence if $Z\cap K\neq\emptyset$ then $%
\mu(B(Z,r))\le r^{d_\mu-\epsilon}$.

Given an integer $n$, let 
\begin{equation*}
W(r,k,n)=\{\omega \in \Omega \colon T^{n}(\omega )\in B(\omega,r) \text{ and 
}\Vert \alpha S_{n}\varphi (\omega )\Vert <r\}.
\end{equation*}%
Recall that $\varphi=1_I $ is constant on $m$-cylinders for some integer $%
m\geq 1$. Assume that $n\ge m+k$ and let us decompose its Birkhoff sum as 
\begin{equation*}
S_{n}\varphi =S_{k-m}\varphi +S_{m}\varphi \circ T^{k-m}+S_{n-k-m}\varphi
\circ T^{k}+S_{m}\varphi \circ T^{n-m}.
\end{equation*}%
Denote by $E$ the range of $S_{m}\varphi $. We have $\card E\leq
p^{2m}<\infty $. The sum $S_{k-m}\varphi $ is constant on $Z$ and we denote
its common value by $q_{Z}$. Then for any $\omega $ in $Z$ we have $%
S_{n}\varphi (\omega )=q_{Z}+u+S_{n-k-m}\varphi (T^{k}\omega )+v$ for some $%
u,v\in E$. Notice that $S_{n-k-m}\varphi (\omega )$ is $\sigma (T^{-k}%
\mathcal{J}_{n-k})$ measurable, thus the $\psi $-mixing property\footnote{%
The $\psi $-mixing property is the following: there exists some sequence $%
\Psi (n)\searrow 0$ such that for any $k$, any $A\in \sigma (\mathcal{J}%
_{k}) $ and any measurable set $B$ we have 
\begin{equation*}
|\mu (A\cap T^{-k-n}B)-\mu (A)\mu (B)|\leq \Psi (n)\mu (A)\mu (B).
\end{equation*}%
} of the measure $\mu $ implies that 
\begin{equation*}
\begin{split}
\mu (K\cap W(r,k,n))& =\sum_{Z\in \mathcal{J}_{k}}\mu (Z\cap K\cap W(r,k,n))
\\
& \leq \Psi (0)^{2}\sum_{Z\in \mathcal{J}_{k},Z\cap K\neq \emptyset
}\sum_{u,v\in E} \\
&\quad\quad \mu (Z\cap T^{-k}\{\Vert \alpha (q_{Z}+u+v+S_{n-k-m}\varphi
)\Vert <r\}\cap T^{-n}B(Z,r)) \\
& \leq \Psi (0)^{2}\sum_{Z\in \mathcal{J}_{k},Z\cap K\neq \emptyset
}\sum_{u,v\in E} \\
&\quad\quad \mu (Z)\mu (\Vert \alpha (q_{Z}+u+v+S_{n-k-m}\varphi )\Vert
<r)\mu (B(Z,r)) \\
& \leq \Psi (0)^{2}p^{4m}r^{d_{\mu }-\epsilon }\left( D_{n-k-m}^{\mu
}(\alpha )+(2r)^d\right) .
\end{split}%
\end{equation*}%
Take $\delta \in (0,\min (d_{\mu },\Delta ))$. When $r^{-\delta }\leq n\leq
r^{-\Delta }$ we have, for some constant $c$, 
\begin{equation*}
\mu (K\cap W(r,k,n))\leq c r^{d_{\mu }-\epsilon } \left[n^{\frac{-1}{2\gamma 
}}+r^d\right].
\end{equation*}%
Therefore 
\begin{equation*}
\sum_{r^{-\delta }\leq n\leq r^{-\Delta }}\mu (K\cap W(r,k,n)) = O\left(
cr^{d_{\mu }-\epsilon }\left[(r^{-\Delta })^{1-1/2\gamma }+r^d\right]%
\right)=O(r^{\epsilon }),
\end{equation*}%
by our choice of $\Delta $. This shows that the probability that for some $%
\omega \in K$ there is a $r$-return of $x=(\omega ,t)$ between the times $%
r^{-\delta }$ and $r^{-\Delta }$ is $O(r^{\epsilon })$. By a Borel Cantelli
argument, for $\mu $-a.e. $\omega \in K$, there are no returns in this time
interval. On the other hand, by Theorem~\ref{maine} the recurrence rate for
the base map $T$ only is equal to $d_{\mu }>\delta $, therefore there are no
returns in the time interval $1,\ldots ,r^{-\delta }$ also. The result
follows since $\epsilon $ is arbitrary.
\end{proof}

\subsection{Observed systems}

In the previous sections, we studied the return time and hitting time of the
skew-product and we showed that their exponents depend on the arithmetical
properties of $\alpha$. A natural question is to study only the return of
the second coordinate instead of the whole system. We are going to study the
Poincar\'e recurrence for a specific observation~\cite{SR}, the canonical
projection $\pi_2: \Omega\times\mathbb{T}^d\rightarrow\mathbb{T}^d$ and
prove that the recurrences rates for the observation do not depend on $T$.

\begin{definition}
Let $x\in \Omega\times \mathbb{T}^d$, $r>0$ and $p\in \mathbb{N}$. We define
the $p$-non-instantaneous return time for the observation 
\begin{equation*}
\tau _{r,p}^{obs}(x,x)=\inf \left\{ k>p,\,\pi_2\left(S^{k}(x)\right)\in
B\left( \pi_2(x),r\right) \right\} .
\end{equation*}
As previously, we define the lower and upper non-instantaneous recurrence
rates for the observation 
\begin{equation*}
\underline{R}^{obs}(x,x)=\lim_{p\rightarrow \infty }\liminf_{r\rightarrow 0}%
\frac{\log \tau _{r,p}^{obs}(x,x)}{-\log r}
\end{equation*}
and 
\begin{equation*}
\overline{R}^{obs}(x,x)=\lim_{p\rightarrow \infty }\limsup_{r\rightarrow 0 }%
\frac{\log \tau _{r,p}^{obs}(x,x)}{-\log r}.
\end{equation*}
\end{definition}

In the following proposition, we will prove that the recurrence rates for
this observation only depends on the underlying rotation.

\begin{proposition}
For $\nu$-almost every $x\in M$ 
\begin{equation*}
\underline{R}^{obs}(x,x)= \underline{rec}(\alpha) (=\frac{1}{\gamma_s(\alpha)%
})\qquad and \qquad\overline{R}^{obs}(x,x)= \overline{rec}(\alpha).
\end{equation*}
\end{proposition}

\begin{proof}
First of all, we can remark that for every $x=(\omega ,t)\in \Omega \times 
\mathbb{T}^d$, every $r>0$ and every $p\in \mathbb{N^*}$ 
\begin{equation*}
\tau _{r,p}^{obs}(x,x)\geq \tau _{r}(t,t)
\end{equation*}%
where $\tau _{r}(t,t)$ is the return time of $t$ with respect to the
rotation with angle $\alpha $. This gives us that 
\begin{equation*}
\underline{R}^{obs}(x,x)\geq \underline{rec}(\alpha)\qquad \text{and}\qquad 
\overline{R}^{obs}(x,x)\geq \overline{rec}(\alpha).
\end{equation*}

On the other hand, we have 
\begin{eqnarray}
\tau _{r,p}^{obs}(x,x) &=&\inf \left\{ k>p,\,S^{k}(\omega ,t)\in \Omega
\times B\left( t,r\right) \right\}  \notag \\
&=&\inf \left\{ k>p,\,t+\alpha S_{k}\varphi(\omega )\in B\left( t,r\right)
\right\}  \notag \\
&=&\inf \left\{ k>p,\,\Vert \alpha S_{k}\varphi(\omega )\Vert \leq r\right\}.
\label{sombirk}
\end{eqnarray}

Birkhoff ergodic theorem gives for $\mu $-almost every $\omega \in \Omega $ 
\begin{equation*}
\lim_{k\rightarrow +\infty }\frac{1}{k}S_{k}\varphi(\omega )=\mu (I)=:a.
\end{equation*}%
Then, for $\mu $-almost every $\omega \in \Omega $ there exists $N\in 
\mathbb{N} $ such that $\forall k>N$, $S_{k}\varphi(\omega )>\frac{a}{2}k$.

By definition of $\underline{rec}(\alpha)$, there exists a sequence $r_n\to0$
and a sequence of integers $q_n\ge p$ such that $\tau_{r_n}(t,t)\le q_n$ and 
$\lim \frac{q_n}{-\log r_n}=\underline{rec}(\alpha)$. For all $n> Na/2$ we
have $S_{\lceil\frac{2}{a}q_{n}\rceil}\varphi(\omega )>q_{n}$ and thus 
\begin{equation*}
\tau _{r_n,p}^{obs}(x,x)\leq \frac{2}{a}q_{n}.
\end{equation*}%
Letting $n$ goes to infinity and then $p$ to infinity proves that 
\begin{equation*}
\underline{R}^{obs}(x,x) \le \underline{rec}(\alpha).
\end{equation*}
We proceed similarly with $\overline{R}^{obs}$.
\end{proof}

The hitting time of the observed system may be defined similarly,

\begin{definition}
Let $x,y\in \Omega\times \mathbb{T}^d$, $r>0$. We define the hitting time
for the observation 
\begin{equation*}
\tau _{r}^{obs}(x,y)=\inf \left\{ k>0,\,\pi_2\left(S^{k}(x)\right)\in
B\left( \pi_2(y),r\right) \right\} .
\end{equation*}
As previously, we define the lower and upper hitting time exponents for the
observation 
\begin{equation*}
\underline{R}^{obs}(x,y)= \liminf_{r\rightarrow 0}\frac{\log \tau
_{r}^{obs}(x,y)}{-\log r}
\end{equation*}
and 
\begin{equation*}
\overline{R}^{obs}(x,y)= \limsup_{r\rightarrow 0 }\frac{\log \tau
_{r}^{obs}(x,y)}{-\log r}.
\end{equation*}
\end{definition}

Also for the hitting time, with an obvious modification of the proof for
return times, one can show that the exponents do not depend on $T$:

\begin{proposition}
For every $y\in M$ and for $\nu$-almost every $x\in M$ 
\begin{equation*}
\underline{R}^{obs}(x,y)= \underline{hit}(\alpha)\qquad and \qquad\overline{R%
}^{obs}(x,y)= \overline{hit}(\alpha).
\end{equation*}
\end{proposition}

\section{Decay of correlations, lower bounds\label{decorr2}}

In \cite[Corollary 3]{GP??} the following relation between decay of
correlations and hitting time exponent is proved :

\begin{proposition}
\label{decorrangle} If a map on a finite dimensional Riemannian manifold $M$%
, has an absolutely continuous invariant measure with strictly positive
density at some point $y_{0}\footnote{%
The density is greater than some positive number in a neighborhood of $y_{0}$%
.}$ and s.t. $\overline{R}(x,y_{0})=R_{0}$ , $x$-a.e., then the speed of
decay of correlations (over Lipschitz observables) of the system is at most
a power law with lower exponent (see Definition \ref{def:decorr}) 
\begin{equation*}
p=\lim \inf_{n\rightarrow \infty }\frac{-\log \Phi (n)}{\log n}\leq {\frac{%
2\dim M+2}{R_{0}-\dim M}}.
\end{equation*}
\end{proposition}

We remark that the assumptions on the absolute continuity of the invariant
measure can be largely relaxed (see \cite{GP??}). By this proposition and
Proposition \ref{11} we have easily:

\begin{proposition}
The decay of correlations with respect to Lipschitz observables of a skew
product on $M=\Omega\times \mathbb{T}^{d}$, such that $\mu $ is absolutely
continuous with positive density is a power law with exponent 
\begin{equation}
p\leq \frac{2\dim M+2}{\max (\dim M,\gamma _{l}(\alpha ))-\dim M}.
\label{ll}
\end{equation}
\end{proposition}

These bounds are probably not sharp but the advantage is that they clearly
show the dependence on the Diophantine type of $\alpha$. A comparison with
the upper bound given in Proposition \ref{pro:doc} gives that for the skew
product of the doubling map and a circle rotation endowed with the Lebesgue
measure, the exponent $p$ satisfies 
\begin{equation*}
\frac{1}{2\gamma (\alpha )}\leq p\leq \frac{6}{\max (2,\gamma (\alpha ))-2}.
\end{equation*}

\section{Systems with $\protect\underline{R}(x,y)>d$ and consequences}

\subsection{A skew product with lower hitting time exponent larger than the
dimension}

\label{lht}

With a suitable torus translation it is possible to obtain a system \cite%
{GP??} where the \emph{lower} hitting time exponent is bigger than the local
dimension ($\underline{R}(x,y)>d_{\nu }$ for typical $x,y$) this leads to
some other consequence as the lack of a dynamical Borel-Cantelli property,
or the triviality of the limit return time distribution.

Using a rotation of this kind in a skew product we obtain

\begin{theorem}
\label{propsumm}There exists a system which has polynomial decay of
correlation over $C^{r}$ observables, superpolynomial decay w.r.t. $%
C^{\infty }$ones and

\begin{itemize}
\item its lower exponent of decay of correlations with respect to Lipschitz
observables satisfies $\frac{1}{16}\leq \lim \inf_{n\rightarrow \infty }%
\frac{-\log \Phi (n)}{\log n}\leq {\frac{8}{13}}$;

\item the lower hitting time indicator $\underline{R}(x,y)$ is bigger than
the local dimension;

\item it has no Monotone Shrinking Target property;

\item it has trivial limit return time statistics.
\end{itemize}
\end{theorem}

Let us describe this example. Let us consider a rotation $T_{\mathbf{\alpha }%
}$ , $\alpha \mathbf{=}(\alpha _{1},\alpha _{2})$ on the torus $\mathbb{T}%
^{2}\cong \mathbb{R}^{2}/\mathbb{Z}^{2}$ by the angle $\alpha$. 
Suppose $\gamma _{1},\gamma _{2}$ are respectively the types of $\alpha _{1}$
and $\alpha _{2}$ and denote by $q_{n}$ and $q_{n}^{\prime }$ the partial
convergent denominators of $\alpha _{1}$ and $\alpha _{2}$.

To obtain a rotation with large lower hitting time indicator let us take $%
\xi >2$ and let $Y_{\xi }\subset \mathbf{R}^{2}$ be the class of couples of
irrationals $(\alpha _{1},\alpha _{2})$ given by the following conditions on
their convergents to be satisfied eventually: 
\begin{equation*}
q_{n}^{\prime }\geq q_{n}^{\xi };
\end{equation*}%
\begin{equation*}
q_{n+1}\geq q_{n}^{\prime }{}^{\xi }.
\end{equation*}%
We note that each $Y_{\xi }$ is uncountable and dense in $[0,1]\times
\lbrack 0,1]$ and each irrational of the couple is of type at least $\xi
^{2}.$ Indeed if $\alpha _{1}$ is such a number: $q_{n+1}\geq q_{n}^{\xi
^{2}}$ (then $q_{n}\geq q_{1}^{\xi ^{2n}}$) and then $\gamma (\alpha
_{1})\geq \limsup_{n\rightarrow \infty }\frac{\log q_{n}^{\xi ^{2}}}{\log
q_{n}}=\xi ^{2}$. We also remark that $\gamma _{l}(\alpha )\geq \max (\gamma
(\alpha _{i}))\geq \xi ^{2}$.

With some more work  (see Propositions \ref{diofalin1} and \ref{diofalin2} in
Subsection~\ref{sec:intertwined}) it is possible to obtain the following:

\begin{proposition}\label{33}
There is an angle $\alpha \in Y_{4}$ which is of finite Diophantine type for
the linear approximation. More precisely $\gamma _{l}(\alpha )=16$.
\end{proposition}

This implies that the following example exists.

\begin{example}
\label{badexample} Let us consider the skew product $(M,S)$ where $%
M=[0,1]\times $ $\mathbb{T}^{2}$ with $\alpha \in Y_{4},$ such that $\gamma
_{l}(\alpha )=16$ and $T$ preserves an absolutely continuous invariant
measure $\mu $ with strictly positive density.
\end{example}

We will see below that the example satisfies all items of Theorem \ref%
{propsumm}.

The reason we take a rotation with angle\ in $Y_{\xi }$ is that the lower
hitting time indicator is bounded from below by $\xi $.

\begin{theorem}[\protect\cite{GP??}]
\label{thm:infliminf} If $\ T_{(\alpha _{1},\alpha _{2})}$ is a rotation of
the two torus by a vector $(\alpha _{1},\alpha _{2})\in Y_{\xi }$ and $y\in 
\mathbb{T}^{2}$, then for almost every $x\in \mathbb{T}^{2}$ 
\begin{equation*}
\underline{R}(x,y)\geq \xi .
\end{equation*}
\end{theorem}

In particular we remark that if $\xi >2$ then in this example the lower
hitting time indicator is bigger than the local dimension. Using Theorem~\ref%
{11} we get:

\begin{proposition}
In a skew product with a rotation by an angle included in $Y_{\xi }$, at
each target point $y$ it hold 
\begin{equation}
\underline{R}(x,y)\geq \max (\underline{d}_{\mu }(\pi_1(y) )+2,\xi )
\end{equation}%
for $\nu$-a.e. $x$.
\end{proposition}

Now we can obtain the first item in Theorem \ref{propsumm}.

\begin{proposition}
In a skew product on $M=\Omega \times T^{2}$ with a rotation by an angle $%
\alpha \in Y_{4}$ with $\gamma _{l}(\alpha )=16$ as above, the decay of
correlations with respect to Lipschitz observables satisfies 
\begin{equation*}
\frac{1}{16}\leq \liminf_{n\rightarrow \infty }\frac{-\log \Phi (n)}{\log n}%
\leq {\frac{8}{13}}.
\end{equation*}
\end{proposition}

\begin{proof}
The second inequality is obtained putting $d=2,\overline{{R}}\geq 4^{2},$%
(see Theorem~\ref{11}, recalling that if $\alpha =(\alpha _{1},\alpha _{2})$
\ then $\gamma _{l}(\alpha )\geq \max_{i}(\gamma (\alpha _{i}))$) in
Proposition \ref{decorrangle}. The first inequality is obtained by
Proposition \ref{pro:doc}.
\end{proof}

Obviously the requirement on the dimension of $\mu $ is independent from the
others, and such system exists. Hence we have the second item of Theorem \ref%
{propsumm}.

\begin{proposition}
In the above example \underline{$R$}$(x,y)>d_\nu(y)$ almost everywhere.
\end{proposition}

\begin{proof}
Recall that we consider $\beta \in Y_{4}$ \underline{$R$}$(x,y)\geq 4$ for
each $y $ and almost each $x$. Here $d_\nu(y)=3$ almost everywhere.
\end{proof}

In the next subsections we clarify and discuss the other items of the
proposition and finally prove the existence of the angles mentioned in Proposition \ref{33}.

\subsection{No dynamical Borel-Cantelli property and monotone shrinking
targets}

Let us consider a family of balls $B_{i}=B_{r_{i}}(y)$ with $i\in \mathbb{N}$
centered in $y$ and such that $r_{i}\rightarrow 0$. In several systems,
mostly having some sort of fast decay of correlations or generical
arithmetic properties, the following generalization of the second
Borel-Cantelli \ lemma can be proved:%
\begin{equation}
\sum \nu (B_{i})=\infty \Rightarrow \nu (\limsup_{i}T^{-i}(B_{i}))=1
\label{BC}
\end{equation}%
or equivalently $T^{i}(x)\in B_{i}$ for infinitely many $i$, when $x$ is
typical with respect to $\nu $. We recall that there are mixing
systems where the above statement does not hold (an example was given in \cite{B}, nevertheless in this example the speed of decay of
correlations is less than polynomial, see \cite{GP??}). From what it is said
above it easily follows that in our examples (which are polynomially mixing
with respect to $C^{r}$ observables and superpolynomial mixing with respect
to $C^{\infty }$ ones) the above property is also violated hence proving the
second item of Theorem \ref{propsumm}.

\begin{definition}
We say that the system has the monotone shrinking target property if \eqref{BC} holds for every decreasing sequence of balls in $X$ with the same center.
\end{definition}

In \cite{GK} the following fact is proved:

\begin{theorem}
\label{G-K}Assume that there is no atom in $X$, if the system has the
monotone shrinking target property, then for each $y$ we have 
\begin{equation*}
\liminf_{r\rightarrow 0}\frac{\log \tau _{r}(x,y)}{-\log \nu (B_{r}(y))}%
=1\quad \text{for $\nu$-a.e. }x.
\end{equation*}
\end{theorem}

We hence easily have:

\begin{proposition}
In a skew product as above over $M=\Omega \times T^{2}$, if $\mathbf{\alpha }%
\in Y_{\xi }$ and $\xi >d_{\mu }+d$ then the system has not the monotone
shrinking target property.
\end{proposition}

\begin{proof}
Let us consider $y$ such that the local dimension $d_{\mu }(\pi
_{2}(y))=d_{\mu }$ (this is a full measure set). The statement follows from
the definition of local dimension: it holds $\frac{\log \mu (B_{r}(y))}{\log
r}\rightarrow d_{\mu }(\pi _{1}(y))+d$. Moreover

\begin{equation}
\lim \inf_{r\rightarrow 0}(\frac{\log \tau _{B_{r}(y)}(x)}{-\log \mu
(B_{r}(y))}\frac{\log \mu (B_{r}(y))}{\log r})=\underline{R}(x,y)
\end{equation}%
by definition. Since in our examples $\underline{R}(x,y)>d_{\mu }+d$. Then $%
\lim \frac{\log \tau _{B_{r}(y)}(x)}{-\log \mu (B_{r}(y))}>1$ and then by
Theorem \ref{G-K} the system has not the monotone shrinking target property.
\end{proof}

\subsection{Trivial return and hitting time distribution}

The following statement shows that if the logarithm law (hitting time
exponent=dimension ) does not hold then the return time statistic has a
trivial limit (compare with Theorem~\ref{thm:trivialRTS} where the result
holds for a subsequence).

\begin{theorem}
(\cite{G10}) If $(X,T,\nu )$ is a finite measure preserving system over a metric space $X$ and 
\begin{equation}
\underline{R}(x,y)>\overline{d}_{\nu }(y)
\end{equation}%
a.e., then the system has trivial limit return time statistic in sequence $%
B_{r}(y)$. That is, the limit in \eqref{1221} exists for each $t>0$ and $%
g(t)=0 $.
\end{theorem}

The mixing system without logarithm law given in \cite{GP??} hence has
trivial return limit statistic in each centered sequence of balls, this gives
an example of a smooth mixing systems with trivial limit return time
statistics. As said before, this example, as the Fayad example has slower than polynomial
decay of correlations. Since in Example \ref{badexample} we found a system
with polynomial decay of correlations with respect to Lipschitz observables
but lower hitting time indicator is bigger than the local dimension hence
the third item of Theorem \ref{propsumm} is established.

\subsection{Construction of a diophantine angle with intertwined partial
quotients}

\label{sec:intertwined}

We briefly recall the basic definitions and properties of continued
fractions ( for a general reference see e.g. \cite{cassel}) that will be
needed in the sequel. Let $\alpha $ be an irrational number, denote by $%
[a_{0};a_{1},a_{2},\ldots ]$ its continued fraction expansion: 
\begin{equation*}
\alpha =a_{0}+\cfrac{1}{a_1 + \cfrac{1}{a_2 + \ldots}}=:[a_{0};a_{1},a_{2},%
\ldots ].
\end{equation*}%
The integers $a_{0},a_{1},a_{2}\ldots $ are called partial quotients of $%
\alpha $ and are all positive except for $a_{0}$. As usual, we define
inductively the sequences $p_{n}$ and $q_{n}$ by: 
\begin{eqnarray*}
p_{-1}=1, &p_{0}=0,&p_{k+1}=a_{k+1}p_{k}+p_{k-1}\text{ for }k\geq 0; \\
q_{-1}=0, &q_{0}=1,&q_{k+1}=a_{k+1}q_{k}+q_{k-1}\text{ for }k\geq 0.
\end{eqnarray*}%
The fractions $p_{n}/q_{n}$ are called the \emph{convergents} of $\alpha $,
as they do in fact converge to it. Moreover they can be seen as \emph{best
approximations} of $\alpha $ in the following sense. As usual denote by $%
\Vert x\Vert :=\min_{n\in \mathbb{Z}}|x-n|$ the distance of a real number
form the integers. Then $q=q_{n}$ for some $n$ if and only if 
\begin{equation*}
\Vert q\alpha \Vert <\Vert q^{\prime }\alpha \Vert \text{ for every positive 
}q^{\prime }<q
\end{equation*}%
and $p_{n}$ is the integer such that $\Vert q_{n}\alpha \Vert =|q_{n}\alpha
-p_{n}|$.

We will see that there are 2 dimensional angles which are of finite type for
the linear approximation, but have intertwined partial quotients, in a way
that they belong to some nontrivial class $Y_{\xi }$, with $\xi =4$, which
is big enough to have an example which can be used in Section \ref{lht}. The
proof is quite similar to the ones given in \cite[Section 7]{Fa}. Since we
need a slightly different result and we need to change some step we write it
explicitly here.

\begin{proposition}
\label{diofalin1}There are $\alpha $ and $\alpha ^{\prime }$such that their
partial quotients $q_{n},q_{n}^{\prime }$ satisfy for any $n$:

\begin{enumerate}
\item $q_{n-1}^{\prime 4}\leq q_{n}\leq 4q_{n-1}^{\prime 4}$

\item $q_{n}^{4}\leq q_{n}^{\prime }\leq 4q_{n}^{4}$

\item $q_{n}\wedge q_{n-1}^{\prime }=1$

\item $q_{n}^{\prime }\wedge q_{n}=1$
\end{enumerate}
\end{proposition}

where $a\wedge b=1$ means that $a$ and $b$ are relatively prime.

\begin{proof}
We will construct $\alpha $, $\alpha ^{\prime }$ and their partial quotients 
$q_{n},q_{n}^{\prime }$ by constructing the appropriate coefficients $a_{n}$%
, $a_{n}^{\prime }$.

Now let us proceed by induction, assuming that $a_{0},...,a_{n-1}$, $%
a_{0}^{\prime },...,a_{n-1}^{\prime }$ have been constructed satisfying
items 1--4, now let us construct \thinspace $a_{n}$.

Since $q_{n-1}^{\prime }\wedge q_{n-1}=1$ \ then there is an integer $\tau
_{n}<q_{n-1}^{\prime }$ such that%
\begin{equation*}
\tau _{n}q_{n-1}\equiv -q_{n-2}(\func{mod}~q_{n-1}^{\prime }),
\end{equation*}%
so that $q_{n-1}^{\prime }$ divides $\tau _{n}q_{n-1}+q_{n-2}$. Now, choose $%
\rho _{n}$ is such that $\rho _{n}\wedge q_{n-1}^{\prime }=1$ and%
\begin{equation*}
q_{n-1}^{\prime 4}\leq \rho _{n}q_{n-1}\leq 2q_{n-1}^{\prime 4}
\end{equation*}%
and define $a_{n}=\tau _{n}+\rho _{n}$. With this choice of $a_{n}$ 
\begin{equation*}
q_{n}=a_{n}q_{n-1}+q_{n-2}=\rho _{n}q_{n-1}+\tau _{n}q_{n-1}+q_{n-2}.
\end{equation*}%
Since $\tau _{n}\leq q_{n-1}^{\prime }$ we have $q_{n-1}^{\prime 4}\leq
q_{n}\leq 4q_{n-1}^{\prime 4}$ and item 1) is satisfied at step $n$. On the
other hand, from the inductive hypotesis $q_{n-1}\wedge q_{n-1}^{\prime }=1$
and from our choice of $\rho _{n}$ it follows $\rho _{n}q_{n-1}\wedge
q_{n-1}^{\prime }=1$, on the other hand $\tau _{n}q_{n-1}+q_{n-2}$ is a
multiple of $q_{n-1}^{\prime }$. Consequently $q_{n}\wedge q_{n-1}^{\prime
}=1$, and item 3) of the inductive step is proved at step $n$.

Now we construct $a_{n}^{\prime }$ in the same way: $a_{n}^{\prime }=\tau
_{n}^{\prime }+\rho _{n}^{\prime }$ with 
\begin{equation*}
\tau _{n}^{\prime }q_{n-1}^{\prime }\equiv -q_{n-2}^{\prime }(\func{mod}%
~q_{n}),
\end{equation*}%
and $\rho _{n}^{\prime }$ is such that $\rho _{n}^{\prime }\wedge q_{n}=1$
and $q_{n}^{4}\leq \rho _{n}^{\prime }q_{n-1}^{\prime }\leq 2q_{n}^{4}$
(recall that $q_{n}$ was already constructed just above) then%
\begin{equation*}
q_{n}^{\prime }=a_{n}^{\prime }q_{n-1}^{\prime }+q_{n-2}^{\prime }=\rho
_{n}^{\prime }q_{n-1}^{\prime }+\tau _{n}^{\prime }q_{n-1}^{\prime
}+q_{n-2}^{\prime },
\end{equation*}%
and hence%
\begin{equation*}
q_{n}^{4}\leq q_{n}^{\prime }\leq 4q_{n}^{4},
\end{equation*}%
proving Item 2). Finally, we use the already proved relation: $q_{n}\wedge
q_{n-1}^{\prime }=1$ to obtain $q_{n}^{\prime }\wedge q_{n}=1$ (Item 4) at
step $n$ as before.
\end{proof}

\begin{proposition}
\label{diofalin2}The vector $(\alpha ,\alpha ^{\prime })$ described above is
such that 
\begin{equation*}
\gamma _{l}((\alpha ,\alpha ^{\prime }))=16.
\end{equation*}
\end{proposition}

\begin{proof}
As remarked before, it is easy to see that $\gamma (\alpha )=\gamma (\alpha
^{\prime })=\xi ^{2}=16$, hence $\gamma _{l}((\alpha ,\alpha ^{\prime
}))\geq 16$. For the opposite inequality, we have to show that for any $%
(k,l)\in \mathbb{Z}^{2}$, with $\max (|k|,|l|)$ sufficiently large%
\begin{equation*}
||k\alpha +l\alpha ^{\prime }||\geq \frac{1}{(\max (|k|,|l|))^{16}}.
\end{equation*}%
If $k$ or $l=0$ then the problem is reduced to a one dimensional one, and
the exponent is the diophantine exponent of $\alpha $ and $\alpha ^{\prime }$
which was already remarked to be $16$.

Now, let us suppose $l,k\neq 0$ and $\ q_{n-1}^{\prime }\leq \max
(|k|,|l|)\leq q_{n}^{\prime }$ for some $n$.

Let us suppose $q_{n-1}^{\prime }\leq \max (|k|,|l|)\leq q_{n}$ (the case $%
q_{n}\leq \max (|k|,|l|)\leq q_{n}^{\prime }$ will be treated below).

We recall some general properties of continued fractions:%
\begin{eqnarray*}
|\alpha -\frac{p_{n}}{q_{n}}| &\leq &\frac{1}{q_{n}q_{n+1}}, \\
|\alpha ^{\prime }-\frac{p_{n-1}^{\prime }}{q_{n-1}^{\prime }}| &\leq &\frac{%
1}{q_{n-1}^{\prime }q_{n}^{\prime }}.
\end{eqnarray*}

As $|k|\leq q_{n}$ , $|l|\leq q_{n}$%
\begin{eqnarray*}
|k\alpha +l\alpha ^{\prime }-k\frac{p_{n}}{q_{n}}-l\frac{p_{n-1}^{\prime }}{%
q_{n-1}^{\prime }}| &\leq &\frac{1}{q_{n+1}}+\frac{q_{n}}{q_{n}^{\prime
}q_{n-1}^{\prime }} \\
&\leq &\frac{1}{q_{n}^{16}}+\frac{q_{n}}{q_{n}^{4}q_{n-1}^{\prime }} \\
&=&o(\frac{1}{q_{n}q_{n-1}^{\prime }}).
\end{eqnarray*}%
for $n$ (and consequently $\max (|k|,|l|)$ ) sufficiently large. On the other hand,
since $q_{n}\wedge q_{n-1}^{\prime }=1$ and $q_{n}\wedge p_{n}=1$, $k\leq
q_{n}$ implies%
\begin{equation*}
||k\frac{p_{n}}{q_{n}}-l\frac{p_{n-1}^{\prime }}{q_{n-1}^{\prime }}||\geq 
\frac{1}{q_{n-1}^{\prime }q_{n}}.
\end{equation*}%
With the above estimation we have that for large $n$%
\begin{equation*}
||k\alpha -l\alpha ^{\prime }||\geq \frac{1}{2q_{n-1}^{\prime }q_{n}}
\end{equation*}%
thus using again the inequalities between the various $q_{n}$ and $%
q_{n}^{\prime }$ we have%
\begin{equation*}
||k\alpha -l\alpha ^{\prime }||\geq \frac{1}{8q_{n-1}^{\prime 5}}.
\end{equation*}%
Since $q_{n-1}^{\prime }\leq \max (|k|,|l|)$ we obtain%
\begin{equation}\label{eqgamma162}
||k\alpha -l\alpha ^{\prime }||\geq \frac{1}{8(\max (|k|,|l|))^{5}}.
\end{equation}

Now let us consider the case $q_{n}\leq \max (|k|,|l|)\leq q_{n}^{\prime }$:

again%
\begin{eqnarray*}
|\alpha -\frac{p_{n}}{q_{n}}| &\leq &\frac{1}{q_{n}q_{n+1}}, \\
|\alpha ^{\prime }-\frac{p_{n}^{\prime }}{q_{n}^{\prime }}| &\leq &\frac{1}{%
q_{n}^{\prime }q_{n+1}^{\prime }}.
\end{eqnarray*}

As $|k|\leq q_{n}^{\prime }$ , $|l|\leq q_{n}^{\prime }$%
\begin{eqnarray*}
|k\alpha +l\alpha ^{\prime }-k\frac{p_{n}}{q_{n}}-l\frac{p_{n}^{\prime }}{%
q_{n}^{\prime }}| &\leq &\frac{q_{n}^{\prime }}{q_{n}q_{n+1}}+\frac{1}{%
q_{n+1}^{\prime }} \\
&\leq &\frac{1}{q_{n}^{^{\prime }8}}+\frac{q_{n}^{\prime }}{q_{n}^{\prime
4}q_{n}} \\
&=&o(\frac{1}{q_{n}q_{n}^{\prime }}).
\end{eqnarray*}%
On the other hand, since $q_{n}\wedge q_{n-1}^{\prime }=1$ and $q_{n}\wedge
p_{n}=1$, $k\leq q_{n}^{\prime }$ implies%
\begin{equation*}
||k\frac{p_{n}}{q_{n}}-l\frac{p_{n}^{\prime }}{q_{n}^{\prime }}||\geq \frac{1%
}{q_{n}^{\prime }q_{n}}.
\end{equation*}%
With the above estimation we have that for large $n$%
\begin{equation*}
||k\alpha -l\alpha ^{\prime }||\geq \frac{1}{2q_{n}^{\prime }q_{n}}
\end{equation*}%
thus using again the inequalities between the various $q_{n}$ and $%
q_{n}^{\prime }$ we have%
\begin{equation*}
||k\alpha -l\alpha ^{\prime }||\geq \frac{1}{8q_{n}^{5}}.
\end{equation*}%
Since $q_{n}\leq \max (|k|,|l|)$ we obtain%
\begin{equation}\label{eqgamma161}
||k\alpha -l\alpha ^{\prime }||\geq \frac{1}{8(\max (|k|,|l|))^{5}}.
\end{equation}
Finally, \eqref{eqgamma162} together with \eqref{eqgamma161} imply that for any $%
(k,l)\in \mathbb{Z}^{2}$, with $\max (|k|,|l|)$ sufficiently large%
\begin{equation*}
||k\alpha -l\alpha ^{\prime }||\geq\frac{1}{8(\max (|k|,|l|))^{5}}\geq \frac{1}{(\max (|k|,|l|))^{16}}.
\end{equation*}%

\end{proof}

\section{Appendix}

\subsection{Changing regularity of the observables}

We prove the lemma which allows to give a better bound on the decay of
correlations, by passing through higher regularity observables.

\begin{proof}[Proof of Lemma \protect\ref{cr}]
Let $k\le p$ and $\ell \le q$. Let $f\in C^{k}$ and $g\in C^{\ell}$ be
Lipschitz observables such that $\int fd\mu =\int gd\mu =0$. Let us consider
a regularization by convolution. As usual, let us consider a function $\rho
\in C^{\infty }(\mathbb{R}^{d})$, $\rho \geq 0$ having support in $B_{1}(0)$
and such that $\int \rho dx=1.$ Let $\epsilon>0$. Then consider $\rho
_{\epsilon }(x)=\epsilon ^{-d}\rho (\frac{x}{\epsilon })$, this function has support
in $B_{\epsilon }(0)$ and still $\int \rho _{\epsilon }dx=1.$ Let us
consider a multiindex $\alpha =(\alpha _{1},...,\alpha _{d})\in \mathbb{N}%
^{d}$, then remark that 
\begin{equation*}
D^{\alpha }\rho _{\epsilon }=\epsilon ^{-d}D^{\alpha }\left(\rho \left(\frac{%
x}{\epsilon }\right)\right)=\epsilon ^{-d}\epsilon ^{-|\alpha
|}\left(D^{\alpha }\rho\right) \left(\frac{x}{\epsilon}\right)
\end{equation*}%
hence%
\begin{equation*}
\|D^{\alpha }\rho _{\epsilon }\|_{L^{1}}=\int |D^{\alpha }\rho _{\epsilon
}|dx=\epsilon ^{-|\alpha |}\int |D^{\alpha }\rho |dx.
\end{equation*}%
Now consider $f_{\epsilon }$ defined by%
\begin{equation*}
f_{\epsilon }(x)=f\ast \rho _{\epsilon }(x)=\int f(y)\rho (x-y)dy
\end{equation*}%
and $g_{\epsilon }=g\ast \rho _{\epsilon }$ (the convolution of $g$ and $%
\rho_\epsilon$ with respect to Lebesgue measure on $\mathbb{R}^d$). For any $%
\beta\le \alpha $ such that $|\beta|\le k$ we get 
\begin{equation*}
D^{\alpha }f_{\epsilon }(x)=\int D^{\beta}f(y)~D^{\alpha
-\beta}\rho_\epsilon (x-y)dy\leq\| f\| _{C^k}\| D^{\alpha -\beta}\rho
_{\epsilon }\| _{L^{1}}.
\end{equation*}%
This implies that%
\begin{eqnarray*}
\| f_{\epsilon }\| _{C^{p}}&\leq& \| f\| _{C^k}\sup_{|\alpha |\leq p-k}\|
(D^{\alpha}\rho _{\epsilon })\| _{L^{1}} \\
&\leq & \| f\| _{C^k} \sup_{|\alpha |\leq k}\epsilon ^{-(p-k) }\int
|D^{\alpha}(\rho )|dx \\
&\leq&C\epsilon ^{-(p-k)}~\| f\| _{C^k}
\end{eqnarray*}%
where $C$ is a constant depending on the function $\rho$.

Moreover, for any $\epsilon$ and $\delta$, it holds that $||f-f_{\epsilon
}||_{\infty }\leq \epsilon^k ~||f||_{C^p}$ , $||g-g_{\delta }||_{\infty
}\leq \delta^\ell ~||g||_{C^\ell}$. Now let us estimate the decay of
correlations of $f$ and $g$ by their regularized functions:%
\begin{eqnarray*}
& &\left|\int f\circ T^{n}gd\nu \right| \\
&\leq& \left|\int (f\circ T^{n}+f_{\epsilon }\circ T^{n}-f_{\epsilon }\circ
T^{n})~(g+g_{\delta }-g_{\delta })d\nu \right| \\
&\leq &\int \left|(f\circ T^{n}-f_{\epsilon }\circ T^{n})~(g-g_{\delta
})\right|d\nu +\int \left|(f\circ T^{n}-f_{\epsilon }\circ T^{n})~(g_{\delta
})\right|d\nu \\
& &+\int |(f_{\epsilon }\circ T^{n})~(g-g_{\delta })|d\nu +\int
|(f_{\epsilon }\circ T^{n})~(g_{\delta })|d\nu \\
&\leq&\epsilon^k\delta^\ell~||f||_{C^k}~||g||_{C^\ell}+\epsilon^k
~||f||_{C^k}~||g||_{1}+\delta^\ell ~||f||_{1}~||g||_{C^\ell}+\| f_{\epsilon
}\| _{C^{p}}\| g_{\delta }\| _{C^{q}}\Phi_{k,\ell} (n) \\
&\leq& \epsilon^k\delta^\ell~||f||_{C^k}~||g||_{C^\ell}+\epsilon^k
~||f||_{C^k}~||g||_{1}+\delta^\ell ~||f||_{1}~||g||_{C^\ell}
+C\epsilon^{k-p}\delta^{q-\ell}~\| f\| _{C^k}\| g\| _{C^\ell}~\Phi_{p,q} (n)
\\
&\leq&
||f||_{C^p}~||g||_{C^q}(\epsilon^k\delta^\ell+(\epsilon^k+\delta^\ell) ~(1+%
\text{diam}(X))+C\epsilon^{k-p}\delta^{\ell-q}\Phi_{p,q} (n)).
\end{eqnarray*}

since $||f||_{1}\leq \Vert f\Vert _{C^p}(1+\text{diam}(X))$. This will be
essentially minimized when $\epsilon^k=\delta^\ell=\epsilon^{k-p}\delta^{%
\ell-q}\Phi_{p,q} (n)$. This gives 
\begin{equation*}
\left\vert \int f\circ T^{n}gd\nu \right\vert \leq
||f||_{C^k}~||g||_{C^\ell} \Phi (n)^{\frac{k\ell}{p\ell+qk-k\ell} }.
\end{equation*}

In the case $k=1$ or $\ell=1$ the same estimate is valid for Lipschitz
observables\footnote{%
We recall that, by Rademacher Theorem, a Lipschitz function has derivatives
defined almost-everywhere and the derivatives are a.e. bounded by the
Lipschitz constant of the function. This clearly suffices to make the proof
works.} with the $C^1$ norm replaced by the Lipschitz one.
\end{proof}

\subsection{Discrepancy estimates}

We prove here Proposition~\ref{pro:dnmu}.

\begin{proof}
We follow the line of \cite{su} (Theorem 5.5) who established this result in
the case of the simple symmetric random $\pm\alpha$ on the circle (that is $%
d=1$ there and the increments where $\pm\alpha$ with independent equal
probability).

The discrepancy of a probability distribution $Q$ on the torus is defined by
the maximal difference, among all rectangles $R$ of $\mathcal{P}^d$, between 
$Q(R)$ and its Lebesgue measure. Recall Erd\"os-T\'uran-Koksma inequality 
\cite{niederreiter} who estimates the discrepancy in terms of the Fourier
transform $\Phi_Q$: 
\begin{equation}  \label{eq:ETK}
D(Q,Leb) \le 3^s \left(\frac{2}{H+1} + \sum_{0<|h|_\infty\le H} \frac{1}{r(h)%
} \left|\Phi_Q(h)\right|\right),
\end{equation}
where $r(h)=\prod_{i=1}^d\max(1,|h_i|)$. Let $D=\{1,\ldots,d\}$. We apply it
with $Q$ the distribution of $\alpha S_N\varphi$ under $\mu$ which satisfies 
\begin{equation*}
|\Phi_Q(h) | = \left|\int_\Omega e^{-2i\pi\langle h,\alpha
S_N\varphi\rangle}d\mu\right|= \left|\int_\Omega L_{2\pi \langle h,
\alpha\rangle}^N(1)d\mu\right| \le c_0e^{-c_1 N \langle h,\alpha \rangle^2},
\end{equation*}
by Proposition~\ref{pro:perturbation}.

Let $\gamma>\gamma_l(\alpha)$. There exists some constant $c_2$ such that
for any $h\neq0$, $\|\langle h,\alpha\rangle\|> c_2|h|^{-\gamma}$.

For $K\subset D$ and $H\in\mathbb{N}$ let 
\begin{equation*}
H^K=\{h\in\mathbb{N}^d\colon 0\le h_i\le H \text{ if }i\in K, h_i=0 \text{
otherwise}\}
\end{equation*}
and 
\begin{equation*}
\Gamma_N(H^K)=\sum_{0\neq h\in H^K} \frac{1}{r(h)} e^{-c_1N\langle
h,\alpha\rangle^2}.
\end{equation*}
We will show that 
\begin{equation}  \label{eq:hgam}
\sum_{0<|h|_\infty\le H} \frac{1}{r(h)} \left|\Phi_Q(h)\right| = O (\frac{%
H^{\gamma-1}}{\sqrt{N}}).
\end{equation}
According to \eqref{eq:ETK}, the proposition will follow from \eqref{eq:hgam}
with the choice $H=N^{\frac{1}{2\gamma}}$. Clearly it suffices to show that
for $K=D$ 
\begin{equation}  \label{eq:hk}
\Gamma_N(H^K)=O (\frac{H^{\gamma-1}}{\sqrt{N}}).
\end{equation}
We prove it by induction on the cardinality of $K\subset D$. For $%
K=\emptyset $ there is nothing to prove. Let $\ell\le d-1$ and assume that
it is true for any $K\subset D$ of cardinality $\ell$. Let $K\subset D$ with
cardinality $\ell+1$. The sum for $h\in H^K$ restricted a face $h_i=0$, for
some $i\in K$, corresponds to the $\ell$-dimensional situation, where the
estimate holds. Therefore, it suffices to prove the estimate for the sum
restricted to $h_i\ge 1$ for any $i\in K$. In this case $r(h)$ is simply the
product $\prod_{i\in K} h_i$.

Given $a\colon \mathbb{N}^d\to\mathbb{R}$, $n\in\mathbb{N}^d$ and $i\in D$
let $a_i(n)=a(n)-a(n+e_i)$, $a^i(n) = \sum_{k_i=1}^{n_i}a(n+(k_i-n_i)e_i)$.
We extend this notation to multiindices $I\subset\{1,\ldots,d\}$, defining $%
a^I(n)$ and $a_I(n)$ in the obvious way. Note $n-I = n-\sum_{i\in I} e_i$.
We use the following multidimensional Abel's summation formula, for any $%
a,b\colon\mathbb{N}^d\to\mathbb{R}$: 
\begin{equation}  \label{eq:abel}
(ab)^K = \sum_{I\subset K} (b_I a^K)^I(n-I).
\end{equation}
We apply this formula to $b(h)=\prod_{i\in K} h_i^{-1}$ and $a(h)=e^{-c_1 N
\|\langle h,\alpha \rangle\|}$.

We first estimate $a^K(h)$: 
\begin{equation*}
a^K(h) = \sum_{1\le k\le h} e^{-c_1N\|\langle k,\alpha\rangle\|},
\end{equation*}
where $1\le k\le h$ means that $\forall i\in K$, $1\le k_i\le h_i$ and $%
k_i=0 $ otherwise. Take $\eta=c_0|2h|_\infty^{-\gamma}$. Whenever $0\le
k\neq k^{\prime }\le h$ we have 
\begin{equation*}
\big|\|\langle k,\alpha\rangle\|-\|\langle k^{\prime },\alpha\rangle\| \big|%
\ge \min( \|\langle k\pm k^{\prime },\alpha\rangle\|)> \eta.
\end{equation*}
Therefore, each of the interval $[0,\eta)$, $[\eta,2\eta)$, ..., contains at
most one $\| \langle k,\alpha\rangle\|$ with $1\le k\le h$, and the first
interval does not contain any of them since $\langle 0,\alpha\rangle\in
[0,\eta)$. Thus, setting $c=c_1 c_2$, 
\begin{equation*}
a^K(h)\le \sum_{j=1}^\infty e^{-c N (j\eta)^2} \le \int_0^\infty e^{-c
N\eta^2 u^2}du = \frac{1}{\eta\sqrt{N}}\int_0^\infty e^{-cv^2}dv,
\end{equation*}
with the change of variable $v=\eta u\sqrt{N}$. That is 
\begin{equation*}
a^K(h) \le C \frac 1{\sqrt{N}}|h|_\infty^{\gamma}
\end{equation*}
for some constant $C$. On the other hand, we have $b_I(h) = b(h) \prod_{i\in
I} (h_i+1)^{-1}$. This implies that, for some constant $C^{\prime }$ and any 
$I\subset K$, 
\begin{equation*}
(b_I a^K)^I (n^{I,K,H})\le C^{\prime }\frac 1{\sqrt{N}} H^{\gamma-1},
\end{equation*}
where $n^{I,K,H}_i=H-1,H,0$ according to $i\in I,K,D\setminus K$. The
inducing step \eqref{eq:hk} follows by \eqref{eq:abel}.
\end{proof}

\end{document}